\RequirePackage{fix-cm}
\documentclass[smallextended]{svjour3}       % onecolumn (second format)
\smartqed  % flush right qed marks, e.g. at end of proof
\usepackage{graphicx}
\usepackage{latexsym, amsmath, amsfonts, amscd, amssymb, verbatim, longtable}
\usepackage{multirow, booktabs}
\usepackage{color}
\begin{document}
	
%\newtheorem{Theorem}{Theorem}[section]
%\newtheorem{Proposition}[Theorem]{Proposition}
%---
%\newtheorem{Proposition}{Proposition}[section]
%---
%\newtheorem{Remark}[Theorem]{Remark}
%---
%\newtheorem{Remark}{Remark}[section]
%---
%\newtheorem{Lemma}{Lemma}[section]
%---
%\newtheorem{Lemma}[Theorem]{Lemma}
%---
%\newtheorem{Corollary}{Corollary}[section]
%---
%\newtheorem{Corollary}[Theorem]{Corollary}
%---
%\newtheorem{Definition}{Definition}[section]
%---
%\newtheorem{Definition}[Theorem]{Definition}
%---
%\newtheorem{Example}{Example}[section]
%---
%\newtheorem{Example}[Theorem]{Example}
%---

\title{Qualitative Properties of the Minimum Sum-of-Squares Clustering Problem}


\author{Tran Hung Cuong         \and 
        Jen-Chih Yao \and Nguyen Dong Yen
}


\institute{Tran Hung Cuong
	\at Department of Computer Science, Faculty of Information Technology, Hanoi University of  Industry, 298 Cau Dien Road, Bac Tu Liem District, Hanoi, Vietnam \\
    \email{tranhungcuonghaui@gmail.com}           %  \\
    \and Jen-Chih Yao \at Center for General Education, China Medical University, Taichung 40402, Taiwan\\
    \email{yaojc@mail.cmu.edu.tw}
    \and
    Nguyen Dong Yen\at
    Institute of Mathematics Vietnam Academy of Science and Technology, 18 Hoang Quoc Viet, Hanoi 10307, Vietnam \\
    \email{ ndyen@math.ac.vn}
}

\date{Received: date / Accepted: date}

\maketitle

\begin{abstract}
A series of basic qualitative properties of the minimum sum-of-squares clustering problem are established in this paper. Among other things, we clarify the solution existence, properties of the global solutions, characteristic properties of the local solutions, locally Lipschitz property of the optimal value function, locally upper Lipschitz property of the global solution map, and the Aubin property of the local solution map. We prove that the problem in question always has a global solution and, under a mild condition, the global solution set is finite and the components of each global solution can be computed by an explicit formula. Based on a newly introduced concept of nontrivial local solution, we get necessary conditions for a system of centroids to be a nontrivial local solution. Interestingly, we are able to show that these necessary conditions are also sufficient ones. Finally, it is proved that the optimal value function is locally Lipschitz, the global solution map is locally upper Lipschitz, and the local solution map has the Aubin property, provided that the original data points are pairwise distinct. Thanks to the obtained complete characterizations of the nontrivial local solutions, one can understand better the performance of not only the $k$-means algorithm, but also of other solution methods for the minimum sum-of-squares clustering problem.
\keywords{Clustering problem \and solution existence \and local solution \and optimality condition \and stability property}
\end{abstract}

\section{Introduction}
\label{intro}
Clustering is an important task in data mining and  it is a powerful tool for automated analysis of data. Cluster is a subset of the data set. The elements of a cluster are similar in some sense (see, e.g., \cite[p.~32]{Aggarwal 2014} and \cite[p.~250]{Kantardzic 2011}). 

\medskip

Clustering is an unsupervised technique dealing with problems of organizing a collection of patterns into clusters based on similarity. Cluster analysis  is applied in  different areas such as image segmentation, information retrieval, pattern recognition, pattern classification, network analysis, vector quantization and data compression, data mining and knowledge discovery business, document clustering and image processing (see, e.g., \cite[p.~32]{Aggarwal 2014} and \cite{Kumar 2017}).  

\medskip
There are many kinds of clustering problems, where different criteria are used. Among these criteria, \textit{the \textbf{M}inimum \textbf{S}um-of-\textbf{S}quares \textbf{C}lustering} (MSSC for short) criterion is one of the most used \cite{Bock_1998,Brusco 2006,Costa 2017,Du Merle 2000,Kumar 2017,PhamDinh_LeThi_2009,Peng 2005,Shereli 2005}. Biding by this criterion, one tries to make the sum of the squared Euclidean distances from each data point to the centroid of its cluster as small as possible. The \textit{MSSC problem} requires to partition a finite data set into a given number of clusters in order to minimize the just mentioned sum. 

\medskip
The importance of the MSSC problem was noticed by researchers long time ago and they have developed many algorithms to solve it (see, e.g., \cite{Bagirov_2008,Bagirov_2014,Bagirov_Rubinov_2003,Bagirov_2011,Bagirov_Yearwood_2006,Likas_2003,Ordin_Bagirov_2015,Xie_2011}, and the references therein). Since this is a NP-hard problem \cite{Aloise 2009,Mahajan_NPHard_2009}, the effective existing algorithms reach at most local solutions. These algorithms may include certain techniques for improving the current data partition to seek better solutions. For example, in \cite{Ordin_Bagirov_2015}, the authors proposed a method to find good starting points that is based on the DCA (Difference-of Convex-functions Algorithms). The latter has been applied to the MSSC problem in \cite{Bagirov_2014,PhamDinh_LeThi_2007,PhamDinh_LeThi_2009}.

\medskip
It is well known that a deep understanding on qualitative properties of an optimization problem is very helpful for its numerical solution. To our knowledge, apart from the fundamental necessary optimality condition given recently by Ordin and Bagirov \cite{Ordin_Bagirov_2015}, qualitative properties of the MSSC problem have not been addressed in the literature until now. 

\medskip
The first aim of the present paper is to prove some basic properties of the above problem. We begin with clarifying the equivalence between the mixed integer programming formulation and the unconstrained nonsmooth nonconvex optimization formulation of the problem, that were given in \cite{Ordin_Bagirov_2015}. Then we prove that the MSSC problem always has a global solution and, under a mild condition, the global solution set is finite and the components of each global solution can be computed by an explicit formula.

\medskip
The second aim of our paper is to characterize the local solutions of the MSSC problem. Based on the necessary optimality condition in DC programming \cite{Dem'yanov_Vasil'ev_1985}, some arguments of \cite{Ordin_Bagirov_2015}, and a newly introduced concept of \textit{nontrivial local solution}, we get necessary conditions for a system of centroids to be a nontrivial local solution. Interestingly, we are able to prove that these necessary conditions are also sufficient ones. Since the known algorithms for solving the MSSC problem focus on the local solutions, our characterizations may lead to a better understanding and further refinements of the existing algorithms. Here, by constructing a suitable example, we investigate the performance of the $k$-means algorithm, which can be considered as a basic solution method for the MSSC problem.  

\medskip
The third aim of this paper is to analyze the changes of the optimal value, the global solution set, and the local solution set of the  MSSC problem with respect to small changes in the data set. Three principal stability properties will be established. Namely, we will prove that the optimal value function is locally Lipschitz, the global solution map is locally upper Lipschitz, and the local solution map has the Aubin property, provided that the original data points are pairwise distinct. 

\medskip
The remainder of the paper consists of three sections. Section 2 describes the MSSC problem and its basic properties. The $k$-means algorithm is recalled in this section. In Section 3, we characterize the local solutions and investigate the performance of the just mentioned algorithm. Section 4 presents a comprehensive stability analysis of the MSSC problem.

\section{Basic Properties of the MSSC Problem}

Let $A = \{{\mathbf{a}}^1,...,{\mathbf{a}}^m\}$ be a finite set of points (representing the data points to be grouped) in the $n$-dimensional Euclidean space $\mathbb R^n$ equipped with the scalar product $\langle \mathbf{x},\mathbf{y}\rangle =\displaystyle\sum_{i=1}^n x_iy_i$ and the norm $\|\mathbf{x}\|=\left(\displaystyle\sum_{i=1}^{n}x_i^2\right)^{1/2}$. Given a positive integer $k$ with $k\leq m$, one wants to partition $A$ into disjoint   subsets $A^1,\dots, A^k,$ called \textit{clusters}, such that a \textit{clustering criterion} is optimized. 

\medskip
If one associates to each cluster $A^j$ a \textit{center} (or \textit{centroid}), denoted by ${\mathbf{x}}^j\in \mathbb R^n$, then the following well-known \textit{variance} or \textit{SSQ (Sum-of-Squares) clustering criterion} (see, e.g., \cite[p.~266]{Bock_1998})
\begin{equation*}
\psi(\mathbf{x},\alpha):=\frac{1}{m}\sum_{i=1}^{m}\left(\sum_{j=1} ^{k}\alpha_{ij}\|{\mathbf{a}}^i-{\mathbf{x}}^j\|^2\right)\ \;\longrightarrow\ \;\min,
\end{equation*}
where $\alpha_{ij}=1$  if ${\mathbf{a}}^i\in A^j$ and $\alpha_{ij}=0$ otherwise, is used. Thus, the above partitioning problem can be formulated as the constrained optimization problem
\begin{equation}\label{basic_clustering_problem}\begin{array}{rl}
\min\Big\{\psi(\mathbf{x},\alpha)\; \mid &\in\mathbb R^{n\times k},\ \mathbb{\alpha}=(\alpha_{ij})\in \mathbb{R}^{m\times k},\ \alpha_{ij}\in \{0, 1\},\\ & \displaystyle\sum_{j=1}^{k}\alpha_{ij}=1,\ i=1,\dots,m,\ j=1,\dots, k\Big\},
\end{array}
\end{equation} 
where the centroid system $\mathbf{x}=({\mathbf{x}}^1,\dots,{\mathbf{x}}^k)$ and the incident matrix $\mathbb{\alpha}=(\alpha_{ij})$ are to be found.

\medskip
Since \eqref{basic_clustering_problem} is a difficult \textit{mixed integer programming problem}, instead of it one usually considers (see, e.g., \cite[p.~344]{Ordin_Bagirov_2015}) next \textit{unconstrained nonsmooth nonconvex optimization problem}
\begin{equation}\label{DC_clustering_problem}
\min\Big\{f(\mathbf{x}):=\frac{1}{m}\sum_{i=1}^{m}\left(\min_{j=1,\dots,k} \|{\mathbf{a}}^i-{\mathbf{x}}^j\|^2\right)\, \mid\, \mathbf{x}=({\mathbf{x}}^1,\dots,{\mathbf{x}}^k)\in\mathbb R^{n\times k}\Big\}.
\end{equation} 

Both models \eqref{basic_clustering_problem} and \eqref{DC_clustering_problem} are referred to as \textit{the minimum sum-of-squares clustering problem} (the MSSC problem). As the decision variables of  \eqref{basic_clustering_problem} and \eqref{DC_clustering_problem} belong to different Euclidean spaces, the equivalence between these minimization problems should be clarified. For our convenience, put $I=\{1,\dots,m\}$ and $J=\{1,\dots,k\}$. 

\medskip
Given a vector $\bar {\mathbf{x}}=(\bar {\mathbf{x}}^1,\dots,\bar {\mathbf{x}}^k)\in\mathbb R^{n\times k}$, we inductively construct~$k$ subsets $A^1,\dots,A^k$ of $A$ in the following way. Put $A^0=\emptyset$ and
\begin{equation}\label{natural_clustering}
A^j=\left\{{\mathbf{a}}^i\in A\setminus \left(\bigcup_{p=0}^{j-1} A^p\right)\; \mid\; \|{\mathbf{a}}^i-\bar {\mathbf{x}}^j\|=\displaystyle\min_{q\in J} \|{\mathbf{a}}^i-\bar {\mathbf{x}}^q\|\right\}
\end{equation}
for $j\in J$. This means that, for every $i\in I$, the data point ${\mathbf{a}}^i$ belongs to the cluster $A^j$ if and only if the distance $\|{\mathbf{a}}^i-\bar {\mathbf{x}}^j\|$ is the minimal one among the distances $\|{\mathbf{a}}^i-\bar {\mathbf{x}}^q\|$, $q\in J$, and ${\mathbf{a}}^i$ does not belong to any  cluster $A^p$ with $1\leq p\leq j-1$. We will call this family $\{A^1,\dots,A^k\}$ \textit{the natural clustering associated with $\bar {\mathbf{x}}$}.

\begin{definition}\label{def_attraction_set}
	{\rm  Let $\bar{\mathbf{x}}=(\bar {\mathbf{x}}^1,\dots,\bar {\mathbf{x}}^k)\in\mathbb R^{n\times k}$. We say that the component $\bar {\mathbf{x}}^j$ of $\bar{\mathbf{x}}$ is \textit{attractive} with respect to the data set $A$ if the set
		$$A[\bar {\mathbf{x}}^j]:=\left\{{\mathbf{a}}^i\in A\; \mid\; \|{\mathbf{a}}^i-\bar {\mathbf{x}}^j\|=\displaystyle\min_{q\in J} \|{\mathbf{a}}^i-\bar {\mathbf{x}}^q\|\right\}$$ is nonempty. The latter is called the \textit{attraction set} of $\bar {\mathbf{x}}^j$.}
\end{definition}

Clearly, the cluster $A^j$ in \eqref{natural_clustering} can be represented as follows:
$$A^j=A[\bar {\mathbf{x}}^j]\setminus \left(\bigcup_{p=1}^{j-1} A^p\right).$$

\begin{proposition}\label{equivalence_prop} If $(\bar{\mathbf{x}},\bar{\alpha})$ is a solution of \eqref{basic_clustering_problem}, then $\bar{\mathbf{x}}$ is a solution of \eqref{DC_clustering_problem}. Conversely, if  $\bar{\mathbf{x}}$ is a solution of \eqref{DC_clustering_problem}, then the natural clustering defined by \eqref{natural_clustering} yields an incident matrix $\bar{\alpha}$ such that  $(\bar{\mathbf{x}},\bar{\alpha})$ is a solution of \eqref{basic_clustering_problem}.
\end{proposition}
\begin{proof} First, suppose that  $(\bar{\mathbf{x}},\bar{\alpha})$ is a solution of \eqref{basic_clustering_problem}. As $\psi(\bar{\mathbf{x}},\bar{\alpha})\leq \psi(\bar{\mathbf{x}},\alpha)$ for every $\alpha=(\alpha_{ij})\in \mathbb{R}^{m\times k}$ with  $ \alpha_{ij}\in \{0, 1\}$, $\sum_{j=1}^{k}\alpha_{ij}=1$ for all $i\in I$ and $j\in J$, one must have 
$$\sum_{j=1} ^{k}\bar {\alpha}_{ij}\|{\mathbf{a}}^i-\bar {\mathbf{x}}^j\|^2=\min_{j\in J} \|{\mathbf{a}}^i-\bar {\mathbf{x}}^j\|^2\quad (\forall i\in I).$$ Hence, $\psi(\bar{\mathbf{x}},\bar{\alpha})=f(\bar{\mathbf{x}})$. If $\bar{\mathbf{x}}$ is not a solution of \eqref{DC_clustering_problem}, then one can find some $\tilde{\mathbf{x}}=(\tilde  {\mathbf{x}}^1,\dots,\tilde  {\mathbf{x}}^k)\in\mathbb R^{n\times k}$ such that $f(\tilde{\mathbf{x}})<f(\bar{\mathbf{x}})$. Let $\{A^1,\dots,A^k\}$ be the natural clustering associated with $\tilde{\mathbf{x}}$. For any $i\in I$ and $j\in J$, set  $\tilde {\alpha}_{ij}=1$ if ${\mathbf{a}}^i\in A^j$ and $\tilde {\alpha}_{ij}=0$ if ${\mathbf{a}}^i\notin A^j$. Let $\tilde{\alpha}=(\tilde {\alpha}_{ij})\in \mathbb{R}^{m\times k}$. From the definition of natural clustering and the choice of  $\tilde{\alpha}$ it follows that $\psi(\tilde{\mathbf{x}},\tilde{\alpha})=f(\tilde{\mathbf{x}})$. Then, we have
$$\psi(\bar{\mathbf{x}},\bar{\alpha})=f(\bar{\mathbf{x}})> f(\tilde{\mathbf{x}})=\psi(\tilde{\mathbf{x}},\tilde{\alpha}),$$ contrary to the fact that $(\bar{\mathbf{x}},\bar{\alpha})$ is a solution of \eqref{basic_clustering_problem}.

Now, suppose that $\bar{\mathbf{x}}$ is a solution of \eqref{DC_clustering_problem}. Let 
$\{A^1,\dots,A^k\}$ be the natural clustering associated with $\bar{\mathbf{x}}$. Put $\bar{\alpha}=(\bar {\alpha}_{ij})$, where $\bar {\alpha}_{ij}=1$ if ${\mathbf{a}}^i\in A^j$ and $\bar {\alpha}_{ij}=0$ if ${\mathbf{a}}^i\notin A^j$. It is easy to see that $\psi(\bar{\mathbf{x}},\bar{\alpha})=f(\bar{\mathbf{x}})$. If there is a feasible point $(\mathbf{x},\alpha)$ of \eqref{basic_clustering_problem} such that $\psi(\mathbf{x},\alpha)<\psi(\bar{\mathbf{x}},\bar{\alpha})$ then, by considering the natural clustering $\{\tilde A^1,\dots,\tilde A^k\}$ be associated with $\mathbf{x}$ and letting $\tilde{\alpha}=(\tilde {\alpha}_{ij})$ with $\tilde {\alpha}_{ij}=1$ if ${\mathbf{a}}^i\in\tilde A^j$ and $\tilde {\alpha}_{ij}=0$ if ${\mathbf{a}}^i\notin\tilde A^j$, we have $f(\mathbf{x})=\psi(\mathbf{x},\tilde{\alpha})\leq\psi(\mathbf{x},\alpha)$. Then we get $$f(\mathbf{x})\leq\psi(\mathbf{x},\alpha)<\psi(\bar{\mathbf{x}},\bar{\alpha})=f(\bar{\mathbf{x}}),$$ contrary to the global optimality of $\bar{\mathbf{x}}$ for \eqref{DC_clustering_problem}. One has thus proved that $(\bar{\mathbf{x}},\bar{\alpha})$ is a solution of \eqref{basic_clustering_problem}.
$\hfill\Box$
\end{proof} 

\begin{proposition}\label{natural_clustering_property} If ${\mathbf{a}}^1,...,{\mathbf{a}}^m$ are pairwise distinct points and $\{A^1,\dots,A^k\}$ is the natural clustering associated with a global solution $\bar{\mathbf{x}}$ of \eqref{DC_clustering_problem}, then $A^j$ is nonempty for every $j\in J$.
\end{proposition}
\begin{proof} Indeed, if there is some $j_0\in J$ with  $A^{j_0}=\emptyset$, then the assumption $k\leq m$ implies the existence of an index $j_1\in J$ such that $A^{j_1}$ contains at least two points. Since the elements of $A^{j_1}$ are pairwise distinct, one could find ${\mathbf{a}}^{i_1}\in A^{j_1}$ with ${\mathbf{a}}^{i_1}\neq \bar {\mathbf{x}}^{j_1}$. Setting $\tilde {\mathbf{x}}^j=\bar {\mathbf{x}}^j$ for $j\in J\setminus\{j_0\}$ and  $\tilde {\mathbf{x}}^{j_0}={\mathbf{a}}^{i_1}$, one can easily show that
\begin{equation*}\label{non-optimal}
f(\tilde{\mathbf{x}})-f(\bar{\mathbf{x}})\leq -\frac{1}{m}\|{\mathbf{a}}^{i_1}-\bar {\mathbf{x}}^{j_1}\|^2<0. \end{equation*} This contradicts the assumption saying that $\bar{\mathbf{x}}$ is a global solution of \eqref{DC_clustering_problem}. $\hfill\Box$
\end{proof} 
\begin{theorem}\label{thm_basic_facts}
	Both problems \eqref{basic_clustering_problem}, \eqref{DC_clustering_problem} have solutions.  If ${\mathbf{a}}^1,...,{\mathbf{a}}^m$ are pairwise distinct points, then the solution sets are finite. Moreover, in that case, if $\bar{\mathbf{x}}=(\bar {\mathbf{x}}^1,\dots,\bar {\mathbf{x}}^k)\in\mathbb R^{n\times k}$ is  a global solution of \eqref{DC_clustering_problem}, then the attraction set $A[\bar {\mathbf{x}}^j]$ is nonempty for every $j\in J$ and one has
	\begin{equation}\label{solution_components}
	\bar {\mathbf{x}}^j=\frac{1}{|I(j)|}\displaystyle\sum_{i\in I(j)} {\mathbf{a}}^i,
	\end{equation} where $I(j):=\left\{i\in I\mid {\mathbf{a}}^i\in A[\bar {\mathbf{x}}^j]\right\}$ with $|\Omega|$ denoting the number of elements of $\Omega$.
\end{theorem}
\begin{proof} a) \textit{Solution existence:} By the second assertion of Proposition \ref{equivalence_prop}, it suffices to show that \eqref{DC_clustering_problem} has a solution. Since the minimum of finitely many continuous functions is a continuous function, the objective function of \eqref{DC_clustering_problem} is continuous on $\mathbb R^{n\times k}$. If $k=1$, then the formula for $f$ can be rewritten as $f({\mathbf{x}}^1)=\displaystyle\frac{1}{m}\displaystyle\sum_{i=1}^{m} \|{\mathbf{a}}^i-{\mathbf{x}}^1\|^2$. This smooth, strongly convex function attains its unique global minimum on $\mathbb R^n$ at the point $\bar {\mathbf{x}}^1={\mathbf{a}}^0$, where 
\begin{equation}\label{barycenter_A}
{\mathbf{a}}^0:=\frac{1}{m}\displaystyle\sum_{i\in I} {\mathbf{a}}^i
\end{equation} is the \textit{barycenter} of the data set $A$ (see, e.g., \cite[pp.~24--25]{LeeTamYen_book} for more details). To prove the solution existence of \eqref{DC_clustering_problem} for any $k\geq 2$, put 
$\rho=\displaystyle\max_{i\in I} \|{\mathbf{a}}^i-{\mathbf{a}}^0\|$, where ${\mathbf{a}}^0$ is defined by \eqref{barycenter_A}. Denote by $\bar B({\mathbf{a}}^0,2\rho)$ the closed ball in $\mathbb R^n$ centered at ${\mathbf{a}}^0$ with radius $2\rho$, and consider the optimization problem
\begin{equation}\label{aux_min_problem}
\min\Big\{f(\mathbf{x})\, \mid\, \mathbf{x}=({\mathbf{x}}^1,\dots,{\mathbf{x}}^k)\in \mathbb R^{n\times k},\; {\mathbf{x}}^j\in\bar B({\mathbf{a}}^0,2\rho),\; \forall j\in J\Big\}.
\end{equation} By the Weierstrass theorem, \eqref{aux_min_problem} has a solution $\bar{\mathbf{x}}=(\bar {\mathbf{x}}^1,\dots,\bar {\mathbf{x}}^k)$ with $\bar {\mathbf{x}}^j$ satisfying the inequality $\|\bar {\mathbf{x}}^j-{\mathbf{a}}^0\|\leq 2\rho$ for all $j\in J$. Take an arbitrary point $\mathbf{x}=({\mathbf{x}}^1,\dots,{\mathbf{x}}^k)\in\mathbb R^{n\times k}$ and notice by the choice of $\bar{\mathbf{x}}$ that $f(\bar{\mathbf{x}})\leq f(\mathbf{x})$ if $\|{\mathbf{x}}^j-{\mathbf{a}}^0\|\leq 2\rho$ for all $j\in J$. If there exists at least one index $j\in J$ with $\|{\mathbf{x}}^j-{\mathbf{a}}^0\|> 2\rho$, then denote the set of such indexes by $J_1$ and define a vector $\tilde{\mathbf{x}}=(\tilde {\mathbf{x}}^1,\dots,\tilde {\mathbf{x}}^k)\in\mathbb R^{n\times k}$ by putting $\tilde {\mathbf{x}}^j={\mathbf{x}}^j$ for every $j\in J\setminus J_1$, and $\tilde {\mathbf{x}}^j={\mathbf{a}}^0$ for all $j\in J_1$. For any $i\in I$, it is clear that $\|{\mathbf{a}}^i-\tilde {\mathbf{x}}^j\|=\|{\mathbf{a}}^i-{\mathbf{a}}^0\|\leq \rho<\|{\mathbf{a}}^i-{\mathbf{x}}^j\|$ for every $j\in J_1$, and $\|{\mathbf{a}}^i-\tilde {\mathbf{x}}^j\|=\|{\mathbf{a}}^i-{\mathbf{x}}^j\|$ for every $j\in J\setminus J_1$. So, we have $f(\tilde{\mathbf{x}})\leq f(\mathbf{x})$. As $f(\bar{\mathbf{x}})\leq f(\tilde{\mathbf{x}})$, this yields $f(\bar{\mathbf{x}})\leq f(\mathbf{x})$. We have thus proved that $\bar{\mathbf{x}}$ is a solution of \eqref{DC_clustering_problem}.

b) \textit{Finiteness of the solution sets and formulae for the solution components:} Suppose that ${\mathbf{a}}^1,...,{\mathbf{a}}^m$ are pairwise distinct points,   $\bar{\mathbf{x}}=(\bar {\mathbf{x}}^1,\dots,\bar {\mathbf{x}}^k)\in\mathbb R^{n\times k}$ is a global solution of \eqref{DC_clustering_problem}, and $\{A^1,\dots,A^k\}$ is the natural clustering associated with~$\bar{\mathbf{x}}$. By Proposition \ref{natural_clustering_property}, $A^j\neq\emptyset$ for all $j\in J$. Since 
$$A^j\subset \left\{{\mathbf{a}}^i\in A\mid \|{\mathbf{a}}^i-\bar {\mathbf{x}}^j\|=\displaystyle\min_{q\in J} \|{\mathbf{a}}^i-\bar {\mathbf{x}}^q\|\right\}$$ and $A^j\neq\emptyset$ for every $j\in J$, we see that $|I(j)|\geq 1$ for every $j\in J$. This implies that  right-hand-side of \eqref{solution_components} is well defined for each $j\in J$. To justify that formula, we can argue as follows. Fix any $j\in J$. Since
$$\|{\mathbf{a}}^i-\bar  {\mathbf{x}}^j\|>\min_{q\in J}\|{\mathbf{a}}^i-\bar {\mathbf{x}}^{q}\|\quad \forall i\notin I(j),$$
there exists $\varepsilon>0$ such that \begin{equation}\label{strict_ineq1}
\|{\mathbf{a}}^i-{\mathbf{x}}^j\|>\min_{q\in J}\|{\mathbf{a}}^i-\bar {\mathbf{x}}^{q}\|\quad \forall i\notin I(j)
\end{equation} for any ${\mathbf{x}}^j\in\bar B(\bar {\mathbf{x}}^j,\varepsilon)$. For each ${\mathbf{x}}^j\in\bar B(\bar {\mathbf{x}}^j,\varepsilon)$, put $\tilde{\mathbf{x}}=(\tilde {\mathbf{x}}^1,\dots,\tilde {\mathbf{x}}^k)$ with $\tilde {\mathbf{x}}^q:=\bar {\mathbf{x}}^q$ for every $q\in J\setminus\{j\}$ and $\tilde {\mathbf{x}}^j:={\mathbf{x}}^j$. From the inequality $f(\bar{\mathbf{x}})\leq f(\tilde{\mathbf{x}})$ and the validity of~\eqref{strict_ineq1} we can deduce that 
\begin{equation}\label{basic_evaluations}\begin{array}{rl}
f(\bar{\mathbf{x}})&=\displaystyle\frac{1}{m}\displaystyle\sum_{i=1}^{m}\left(\min_{q\in J} \|{\mathbf{a}}^i-\bar {\mathbf{x}}^q\|^2\right)\\
&=\displaystyle\frac{1}{m}\left[\displaystyle\sum_{i\in I(j)}\|{\mathbf{a}}^i-\bar {\mathbf{x}}^j\|^2+\displaystyle\sum_{i\in I\setminus I(j)}\left(\min_{q\in J} \|{\mathbf{a}}^i-\bar {\mathbf{x}}^q\|^2\right)\right]\\
& \leq f(\tilde{\mathbf{x}})\\
& = \displaystyle\frac{1}{m}\left[\displaystyle\sum_{i\in I(j)}\left(\min_{q\in J} \|{\mathbf{a}}^i-\tilde {\mathbf{x}}^q\|^2\right)+\displaystyle\sum_{i\in I\setminus I(j)}\left(\min_{q\in J} \|{\mathbf{a}}^i-\tilde {\mathbf{x}}^q\|^2\right)\right]\\
& =  \displaystyle\frac{1}{m}\left[\displaystyle\sum_{i\in I(j)}\left(\min_{q\in J} \|{\mathbf{a}}^i-\tilde {\mathbf{x}}^q\|^2\right)+\displaystyle\sum_{i\in I\setminus I(j)}\left(\min_{q\in J} \|{\mathbf{a}}^i-\bar {\mathbf{x}}^q\|^2\right)\right]\\
& \leq\displaystyle\frac{1}{m}\left[\displaystyle\sum_{i\in I(j)}\|{\mathbf{a}}^i-{\mathbf{x}}^j\|^2+\displaystyle\sum_{i\in I\setminus I(j)}\left(\min_{q\in J} \|{\mathbf{a}}^i-\bar {\mathbf{x}}^q\|^2\right)\right].
\end{array}
\end{equation}   Consider the function $\varphi({\mathbf{x}}^j):=\displaystyle\frac{1}{m}\displaystyle\sum_{i\in I(j)}\|{\mathbf{a}}^i-{\mathbf{x}}^j\|^2$, ${\mathbf{x}}^j\in\mathbb R^n$. Comparing the expression on the second line of \eqref{basic_evaluations} with the one on the sixth line yields $\varphi(\bar {\mathbf{x}}^j)\leq \varphi({\mathbf{x}}^j)$ for every ${\mathbf{x}}^j\in\bar B(\bar {\mathbf{x}}^j,\varepsilon)$. Hence $\varphi$ attains its local minimum at $\bar {\mathbf{x}}^j$. By the Fermat Rule we have $\nabla\varphi(\bar {\mathbf{x}}^j)=0$, which gives $\displaystyle\sum_{i\in I(j)}({\mathbf{a}}^i-\bar {\mathbf{x}}^j)=0$. This equality implies \eqref{solution_components}. Since there are only finitely many nonempty subsets $\Omega\subset I$, the set ${\mathcal B}$ of vectors $\mathbf{b}_\Omega$ defined by formula $\mathbf{b}_\Omega=\displaystyle\frac{1}{|\Omega|}\displaystyle\sum_{i\in \Omega} \mathbf{a}^i$ is finite. (Note that  $\mathbf{b}_\Omega$ is the barycenter of the subsystem $\{\mathbf{a}^i\in A\mid i\in\Omega\}$ of $A$.) According to \eqref{solution_components}, each component of a global solution $\bar{\mathbf{x}}=(\bar {\mathbf{x}}^1,\dots,\bar {\mathbf{x}}^k)$  of \eqref{DC_clustering_problem} must belongs to  ${\mathcal B}$, we can assert that the solution set of \eqref{DC_clustering_problem} is finite, provided that ${\mathbf{a}}^1,...,{\mathbf{a}}^m$ are pairwise distinct points.  By Proposition \ref{equivalence_prop}, if $(\bar{\mathbf{x}},\bar{\alpha})$ is a solution of \eqref{basic_clustering_problem}, then $\bar{\mathbf{x}}$ is a solution of \eqref{DC_clustering_problem}. Since $\bar{\alpha}=(\bar {\alpha}_{ij})\in \mathbb{R}^{m\times k}$ must satisfy the conditions $\bar {\alpha}_{ij}\in \{0, 1\}$ and $\displaystyle\sum_{j=1}^{k}\bar {\alpha}_{ij}=1$ for all $i\in I$, $j\in J$, it follows that the solution set of \eqref{basic_clustering_problem} is also finite. $\hfill\Box$
\end{proof}

\begin{proposition}\label{thm_basic_facts(1)} If $\bar{\mathbf{x}}=(\bar {\mathbf{x}}^1,\dots,\bar {\mathbf{x}}^k)\in\mathbb R^{n\times k}$ is a global solution of \eqref{DC_clustering_problem}, then the components of $\bar{\mathbf{x}}$ are pairwise distinct, i.e., $\bar {\mathbf{x}}^{j_1}\neq\bar {\mathbf{x}}^{j_2}$ whenever $j_2\neq j_1$.
\end{proposition}
\begin{proof} On the contrary, suppose that there are distinct indexes $j_1,j_2\in J$ satisfying $\bar {\mathbf{x}}^{j_1}=\bar {\mathbf{x}}^{j_2}$. As $k\leq n$, one has $k-1<n$. So, there must exist $j_0\in J$ such that $|A[\bar {\mathbf{x}}^{j_0}]|\geq 2$. Therefore, one can find a data point ${\mathbf{a}}^{i_0}\in A[\bar {\mathbf{x}}^{j_0}]$ with ${\mathbf{a}}^{i_0}\neq \bar {\mathbf{x}}^{j_0}$. Setting $\tilde{\mathbf{x}}=(\tilde {\mathbf{x}}^1,\dots,\tilde {\mathbf{x}}^k)$ with $\tilde {\mathbf{x}}^j=\bar {\mathbf{x}}^j$ for every $j\in J\setminus\{j_2\}$ and  $\tilde {\mathbf{x}}^{j_2}={\mathbf{a}}^{i_0}$. The construction of $\tilde{\mathbf{x}}$ yields 
\begin{equation*}\label{non-optimal(1)}
f(\tilde{\mathbf{x}})-f(\bar{\mathbf{x}})\leq -\frac{1}{m}\|{\mathbf{a}}^{i_0}-\bar {\mathbf{x}}^{j_0}\|^2<0, \end{equation*} which is impossible because $\bar{\mathbf{x}}$ is a global solution of \eqref{DC_clustering_problem}. $\hfill\Box$
\end{proof} 

\begin{remark}\label{rem_degenerate_case} {\rm If the points ${\mathbf{a}}^1,...,{\mathbf{a}}^m$ are not pairwise distinct, then the conclusions of Theorem \ref{thm_basic_facts} do not hold in general. Indeed, let $A=\{{\mathbf{a}}^1, {\mathbf{a}}^2\}\subset\mathbb R^2$ with ${\mathbf{a}}^1={\mathbf{a}}^2$. For $k:=2$, let $\bar{\mathbf{x}}=(\bar {\mathbf{x}}^1,\bar {\mathbf{x}}^2)$ with $\bar {\mathbf{x}}^1={\mathbf{a}}^1$ and $\bar {\mathbf{x}}^2\in\mathbb R^2$ being chosen arbitrarily. Since $f(\bar{\mathbf{x}})=0$, we can conclude that $\bar{\mathbf{x}}$ is a global solution of \eqref{DC_clustering_problem}. So, the problem has an unbounded solution set. Similarly, for a data set $A=\{{\mathbf{a}}^1,\dots, {\mathbf{a}}^4\}\subset\mathbb R^2$ with ${\mathbf{a}}^1={\mathbf{a}}^2$, ${\mathbf{a}}^3={\mathbf{a}}^4$, and ${\mathbf{a}}^2\neq {\mathbf{a}}^3$. For $k:=3$, let $\bar{\mathbf{x}}=(\bar {\mathbf{x}}^1,\bar {\mathbf{x}}^2,\bar {\mathbf{x}}^3)$ with $\bar {\mathbf{x}}^1={\mathbf{a}}^1$, $\bar {\mathbf{x}}^2={\mathbf{a}}^3$, and $\bar {\mathbf{x}}^3\in\mathbb R^2$ being chosen arbitrarily. By the equality $f(\bar{\mathbf{x}})=0$ we can assert that $\bar{\mathbf{x}}$ is a global solution of \eqref{DC_clustering_problem}. This shows that the solution set of \eqref{DC_clustering_problem} is unbounded. Notice also that, if $\bar {\mathbf{x}}^3\notin\{\bar {\mathbf{x}}^1,\bar {\mathbf{x}}^2\}$, then formula \eqref{solution_components} cannot be applied to $\bar {\mathbf{x}}^3$, because the index set $I(3)=\left\{i\in I\mid {\mathbf{a}}^i\in A[\bar {\mathbf{x}}^3]\right\}=\left\{i\in I\mid \|{\mathbf{a}}^i-\bar {\mathbf{x}}^3\|=\displaystyle\min_{q\in J} \|{\mathbf{a}}^i-\bar {\mathbf{x}}^q\|\right\}$ is empty.}	
\end{remark}

Formula \eqref{solution_components} is effective for computing certain components of any given \textit{local solution} of  \eqref{DC_clustering_problem}. The precise statement of this result is as follows.

\begin{theorem}\label{thm_local_solutions}
	If $\bar{\mathbf{x}}=(\bar {\mathbf{x}}^1,\dots,\bar {\mathbf{x}}^k)\in\mathbb R^{n\times k}$ is a local solution of \eqref{DC_clustering_problem}, then \eqref{solution_components} is valid for all $j\in J$ whose index set $I(j)$ is nonempty, i.e., the component $\bar {\mathbf{x}}^j$ of $\bar{\mathbf{x}}$ is attractive w.r.t. the data set $A$. 
\end{theorem}
\begin{proof}  It suffices to re-apply the arguments described in the second part of the proof of Theorem \ref{thm_basic_facts}, noting that $f(\bar{\mathbf{x}})\leq f(\tilde{\mathbf{x}})$ if ${\mathbf{x}}^j$ (the $j$-th component of $\tilde{\mathbf{x}}$) is taken from $\bar B(\bar {\mathbf{x}}^j,\varepsilon')$ with $\varepsilon'\in (0,\varepsilon)$ being small enough.  $\hfill\Box$
\end{proof} 

As in the proof of Theorem \ref{thm_basic_facts}, if $\Omega=\{{\mathbf{a}}^{i_1},\dots,{\mathbf{a}}^{i_r}\}\subset A$ is a nonempty subset, then we put  $\mathbf{b}_\Omega=\displaystyle\frac{1}{r}\sum_{l=1}^{r}{\mathbf{a}}^{i_l}$. Recall that the set of such points $\mathbf{b}_\Omega$ has been denoted by ${\mathcal B}$.

\begin{remark} {\rm Theorem \ref{thm_basic_facts} shows that if the points ${\mathbf{a}}^1,...,{\mathbf{a}}^m$ are pairwise distinct, then every component of a global solution must belong to ${\mathcal B}$. It is clear that ${\mathcal B}\subset {\rm co}A$, where ${\rm co}A$ abbreviates the \textit{convex hull} of $A$. Looking back to the proof of Theorem~\ref{thm_basic_facts}, we see that the set $A$ lies in the ball $\bar B({\mathbf{a}}^0,\rho)$. Hence ${\mathcal B}\subset {\rm co}A\subset\bar B({\mathbf{a}}^0,\rho)$. It follows that \textit{the global solutions of \eqref{DC_clustering_problem} are contained in the set}
		$$\Big\{\mathbf{x}=({\mathbf{x}}^1,\dots,{\mathbf{x}}^k)\in \mathbb R^{n\times k}\,\mid\, {\mathbf{x}}^j\in\bar B({\mathbf{a}}^0,\rho),\; \forall j\in J\Big\},$$ \textit{provided the points ${\mathbf{a}}^1,...,{\mathbf{a}}^m$ are pairwise distinct}. Similarly, Theorem \ref{thm_local_solutions} assures that each attractive component of a local solution of~\eqref{DC_clustering_problem} belongs to ${\mathcal B}$, where ${\mathcal B}\subset {\rm co}A\subset \bar B({\mathbf{a}}^0,\rho)$.}
\end{remark}

\begin{remark}\label{permutations_are_allowed} {\rm If $\bar{\mathbf{x}}=(\bar {\mathbf{x}}^1,\dots,\bar {\mathbf{x}}^k)\in\mathbb R^{n\times k}$ is a global solution (resp., a local solution) of \eqref{DC_clustering_problem} then, for any permutation $\sigma$ of $J$, the vector $\bar {\mathbf{x}}^{\sigma}:=(\bar {\mathbf{x}}^{\sigma(1)},\dots,\bar {\mathbf{x}}^{\sigma(k)})$ is also  a global solution (resp., a local solution) of \eqref{DC_clustering_problem}. This observation follows easily from the fact that $f(\mathbf{x})=f({\mathbf{x}}^\sigma)$, where $\mathbf{x}=({\mathbf{x}}^1,\dots,{\mathbf{x}}^k)\in\mathbb R^{n\times k}$ and ${\mathbf{x}}^{\sigma}:=({\mathbf{x}}^{\sigma(1)},\dots,{\mathbf{x}}^{\sigma(k)})$.}
\end{remark}

To understand the importance of the above results and those to be established in next two sections, let us recall first the $k$-means clustering algorithm.

\medskip
\textbf{The $k$-means Algorithm}: Despite its ineffectiveness, the $k$-means clustering algorithm (see, e.g., \cite[pp.~89--90]{Aggarwal 2014}, \cite{Jain 2010}, \cite[pp.~263--266]{Kantardzic 2011}, and \cite{MacQueen 1967}) is one of the most popular solution methods for \eqref{DC_clustering_problem}. One starts with selecting $k$ points $ {\mathbf{x}}^1,\dots,{\mathbf{x}}^k$ in $\mathbb R^n$ as the initial centroids. Then one inductively constructs~$k$ subsets $A^1,\dots,A^k$ of the data set $A$ by putting $A^0=\emptyset$ and using the rule \eqref{natural_clustering}, where ${\mathbf{x}}^j$ plays the role of $\bar {\mathbf{x}}^j$ for all $j\in J$. This means that $\{A^1,\dots,A^k\}$ is the natural clustering associated with $\mathbf{x}=({\mathbf{x}}^1,\dots,{\mathbf{x}}^k)$. Once the clusters are formed, for each $j\in J$, if $A^j\neq\emptyset$ then the centroid ${\mathbf{x}}^j$ is updated by the rule 
\begin{equation}\label{new_centroids}
{\mathbf{x}}^j\ \; \leftarrow\ \; \widetilde {\mathbf{x}}^j:=\frac{1}{|I(A^j)|}\, \displaystyle\sum_{i\in I(A^j)} {\mathbf{a}}^i
\end{equation} with $I(A^j):=\left\{i\in I\mid {\mathbf{a}}^i\in A^j\right\}$; and ${\mathbf{x}}^j$ does not change otherwise. The algorithm iteratively repeats the procedure until the centroid system $\{{\mathbf{x}}^1,\dots,{\mathbf{x}}^k\}$ is stable, i.e., $\widetilde {\mathbf{x}}^j={\mathbf{x}}^j$ for all $j\in J$ with  $A^j\neq\emptyset$. The computation procedure is described as follows. 

\medskip
\noindent
\rule[0.05cm]{11.85cm}{0.01cm}\\
\textbf{Input}: The data set $A = \{{\mathbf{a}}^1,...,{\mathbf{a}}^m\}$ and a constant $\varepsilon\geq 0$ (tolerance).\\
\textbf{Output}: The set of $k$ centroids $\{{\mathbf{x}}^1,...,{\mathbf{x}}^k\}$.\\
\textit{Step 1}. Select initial centroids ${\mathbf{x}}^j \in \mathbb{R}^n$ for all $j\in J$.\\
\textit{Step 2}. Compute $ \alpha_i=\min\{\|{\mathbf{a}}^i-{\mathbf{x}}^j\| \mid j\in J\}$ for all $i\in I$.\\
\textit{Step 3}.  Form the clusters $A^1,\dots,A^k$:\\
- Find the attraction sets 
$$A[{\mathbf{x}}^j]=\left\{{\mathbf{a}}^i\in A\; \mid\; \|{\mathbf{a}}^i-{\mathbf{x}}^j\|= \alpha_i\right\}\quad\; (j\in J);$$ 
- Set $A^1=A[{\mathbf{x}}^1]$ and 
\begin{equation}\label{natural_clusters}
A^j=A[{\mathbf{x}}^j]\setminus \left(\bigcup_{p=1}^{j-1} A^p\right)\quad\; (j=2,\dots,k).
\end{equation}
\textit{Step 4}. Update the centroids ${\mathbf{x}}^j$ satisfying $A^j\neq\emptyset$ by the rule \eqref{new_centroids}, keeping other centroids unchanged.\\
\textit{Step 5}. Check the convergence condition: 
If $\|\widetilde{\mathbf{x}}^j-{\mathbf{x}}^j\|\leq\varepsilon$ for all $j\in J$ with $A^j\neq\emptyset$ then stop, else go to \textit{Step 2}.\\
\rule[0.05cm]{11.85cm}{0.01cm}

\medskip
The following example is designed to show how the algorithm is performed in practice.	

\begin{example}\label{Kmeans_example} {\rm Choose $m=3$, $n=2$, and $k=2$. Let $A=\{{\mathbf{a}}^1,{\mathbf{a}}^2,{\mathbf{a}}^3\}$, where ${\mathbf{a}}^1 = (0,0),\, {\mathbf{a}}^2 = (1,0)$, ${\mathbf{a}}^3 = (0,1)$. Apply the $k$-means algorithm to solve the problem \eqref{DC_clustering_problem}  with the tolerance $\varepsilon= 0$. 
		
		(a) With the starting centroids ${\mathbf{x}}^1 ={\mathbf{a}}^1$, ${\mathbf{x}}^2={\mathbf{a}}^2$, one obtains the clusters $A^1=A[{\mathbf{x}}^1]=\{{\mathbf{a}}^1,{\mathbf{a}}^3\}$ and $A^2=A[{\mathbf{x}}^2]=\{{\mathbf{a}}^2\}$. The updated  centroids are ${\mathbf{x}}^1=(0,\frac{1}{2})$, ${\mathbf{x}}^2={\mathbf{a}}^2$. Then, the new clusters $A^1$ and $A^2$ coincide with the old ones. Thus,  $\|\widetilde{\mathbf{x}}^j-{\mathbf{x}}^j\|=0$ for all $j\in J$ with $A^j\neq\emptyset$. So, the computation terminates. For ${\mathbf{x}}^1=(0,\frac{1}{2})$, ${\mathbf{x}}^2={\mathbf{a}}^2$, one has $f(\mathbf{x}) = \frac{1}{6}$.
		
		(b) Starting with the points ${\mathbf{x}}^1 = (\frac{1}{4},\frac{3}{4})$ and ${\mathbf{x}}^2=(2,3)$, one gets the clusters $A^1=A[{\mathbf{x}}^1]=\{{\mathbf{a}}^1, {\mathbf{a}}^2,{\mathbf{a}}^3\}$ and $A^2=A[{\mathbf{x}}^2]=\emptyset$. The algorithm gives the centroid system ${\mathbf{x}}^1=(\frac{1}{3},\frac{1}{3})$, ${\mathbf{x}}^2=(2,3)$, and $f(\mathbf{x}) =\frac{1}{3}$.
		
		(c) Starting with ${\mathbf{x}}^1 = (0,1)$ and ${\mathbf{x}}^2=(0,0)$, by the algorithm we are led to $A^1=A[{\mathbf{x}}^1]=\{{\mathbf{a}}^3\}$, $A^2=A[{\mathbf{x}}^2]=\{{\mathbf{a}}^1, {\mathbf{a}}^2\}$, ${\mathbf{x}}^1=(0,1)$, and ${\mathbf{x}}^2=(\frac{1}{2},0)$. The corresponding value of objective function is $f(\mathbf{x}) = \frac{1}{6}$.
		
		(d) Starting with ${\mathbf{x}}^1 = (0,0)$ and ${\mathbf{x}}^2=(\frac{1}{2},\frac{1}{2})$, by the algorithm one gets the results $A^1=A[{\mathbf{x}}^1]=\{{\mathbf{a}}^1\}$, $A^2=A[{\mathbf{x}}^2]=\{{\mathbf{a}}^2, {\mathbf{a}}^3\}$, ${\mathbf{x}}^1=(0,0)$, and ${\mathbf{x}}^2=(\frac{1}{2},\frac{1}{2})$. The corresponding value of objective function is $f(\mathbf{x}) = \frac{4}{9}$.
		
		(e) With ${\mathbf{x}}^1 = (\frac{1}{3},\frac{1}{3})$ and ${\mathbf{x}}^2=(1+\frac{\sqrt{5}}{3},0)$ as the initial centroids, one obtains the results $A^1=A[{\mathbf{x}}^1]=\{{\mathbf{a}}^1,{\mathbf{a}}^2,{\mathbf{a}}^3\}$, $A^2=A[{\mathbf{x}}^2]=\emptyset$, ${\mathbf{x}}^1=(\frac{1}{3},\frac{1}{3})$, ${\mathbf{x}}^2=(1+\frac{\sqrt{5}}{3},0)$, and $f(\mathbf{x}) = \frac{4}{9}$.}
\end{example}	

\textit{Based on the existing knowledge on the MSSC problem and the $k$-means clustering algorithm, one cannot know whether the five centroid systems obtained in the items} (a)--(e) \textit{of Example \ref{Kmeans_example} contain a global optimal solution of the clustering problem, or not. Even if one knows that the centroid systems obtained in} (a) \textit{and} (c) \textit{are global optimal solutions, one still cannot say definitely whether the centroid systems obtained in the items} (b), (d), (e) \textit{are local optimal solutions of \eqref{DC_clustering_problem}, or not.}

\medskip
The theoretical results in this and the forthcoming sections allow us to clarify the following issues related to the MSSC problem in Example \ref{Kmeans_example}:

- The structure of the global solution set (see Example \ref{basic_example} below);

- The structure of the local solution set (see Example \ref{local_solutions_example});

- The performance of the $k$-means algorithm (see Example \ref{convergence_analysis_example}).

In particular, it will be shown that the centroid systems in (a) and (c) are global optimal solutions, the centroid systems in (b) and (d) are local-nonglobal optimal solutions, while the centroid system in (e) is not a local solution (despite the fact that the centroid systems generated by the $k$-means algorithm converge to it, and the value of the objective function at it equals to the value given by the centroid system in (d)).

\section{Characterizations of the Local Solutions}

In order to study the local solution set of \eqref{DC_clustering_problem} in more details, we will follow Ordin and Bagirov~\cite{Ordin_Bagirov_2015} to consider the problem in light of a well-known optimality condition in DC programming. For every $\mathbf{x}= ({\mathbf{x}}^1,..., {\mathbf{x}}^k)\in\mathbb R^{n\times k}$, we have
\begin{equation}\begin{array}{rl}
f(\mathbf{x}) & =\displaystyle\frac{1}{m}\sum_{i\in I}\left(\min_{j\in J} \|{\mathbf{a}}^i-{\mathbf{x}}^j\|^2\right)\\
& =\displaystyle\frac{1}{m}\sum_{i\in I}\left[\left(\sum_{j\in J}\|{\mathbf{a}}^i-{\mathbf{x}}^j\|^2\right)-\left(\max_{j\in J}\sum_{q\in J\setminus\{j\}}\|{\mathbf{a}}^q-{\mathbf{x}}^j\|^2\right)\right].
\end{array}
\end{equation}
Hence, the objective function $f$ of \eqref{DC_clustering_problem} can be expressed \cite[p.~345]{Ordin_Bagirov_2015} as the difference of two convex functions
\begin{eqnarray}\label{DC function fkx}
f(\mathbf{x}) = f^1(\mathbf{x}) - f^2(\mathbf{x}),
\end{eqnarray} 
where
\begin{eqnarray}\label{DC function fk1}
f^1(\mathbf{x}) := \displaystyle\frac{1}{m}\sum_{i\in I}\left(\sum_{j\in J}\|{\mathbf{a}}^i-{\mathbf{x}}^j\|^2\right)
\end{eqnarray} 
and
\begin{eqnarray}\label{DC function fk2}
f^2(\mathbf{x}) := \frac{1}{m}\sum_{i\in I}\left(\max_{j\in J}\sum_{q\in J\setminus\{j\} }\|{\mathbf{a}}^i-{\mathbf{x}}^q\|^2\right).
\end{eqnarray} 
It is clear that $f^1$ is a convex linear-quadratic function. In particular, it is differentiable. As the sum of finitely many nonsmooth convex functions, $f^2$ is a nonsmooth convex function, which is defined on the whole space $\mathbb R^{n\times k}$. The subdifferentials of $f^1(\mathbf{x})$ and $f^2(\mathbf{x})$ can be computed as follows. First, one has $$\begin{array}{rl}
\partial f^1(\mathbf{x})=\{\nabla f^1(\mathbf{x}) \}& =\left\{\displaystyle\frac{2}{m}\sum_{i\in I} \left({\mathbf{x}}^1-{\mathbf{a}}^i,\dots,{\mathbf{x}}^k-{\mathbf{a}}^i\right)\right\}\\ &= \{2({\mathbf{x}}^1-{\mathbf{a}}^0,\dots,{\mathbf{x}}^k-{\mathbf{a}}^0)\}\end{array}$$
where, as before, ${\mathbf{a}}^0=\mathbf{b}_A$ is the barycenter of the system $\{{\mathbf{a}}^1,\dots,{\mathbf{a}}^m\}$. 
Set
\begin{eqnarray}\label{Phi_i function}
\varphi_i(\mathbf{x}) = \max_{j\in J}\, h_{i,j}(\mathbf{x})
\end{eqnarray} with  $h_{i,j}(\mathbf{x}):=\displaystyle\sum_{q\in J\setminus\{j\}}\|{\mathbf{a}}^i-{\mathbf{x}}^q\|^2$
and
\begin{eqnarray}\label{J_i function}
J_i(\mathbf{x}) = \left\{j\in J\,\mid\, h_{i,j}(\mathbf{x})=\varphi_i(\mathbf{x})\right\}.
\end{eqnarray} 

\begin{proposition}\label{meaning_of_Ji(x)} One has
	\begin{eqnarray}\label{Jix_formula}
	J_i(\mathbf{x}) = \left\{j\in J\,\mid\, {\mathbf{a}}^i\in A[{\mathbf{x}}^j]\right\}.
	\end{eqnarray} 
\end{proposition}
\begin{proof} From the formula of $h_{i,j}(\mathbf{x})$ it follows that
\begin{eqnarray*} h_{i,j}(\mathbf{x})=\left(\displaystyle\sum_{q\in J}\|{\mathbf{a}}^i-{\mathbf{x}}^q\|^2\right)-\|{\mathbf{a}}^i-{\mathbf{x}}^j\|^2.
\end{eqnarray*} Therefore, by \eqref{Phi_i function}, we have
\begin{eqnarray*}\begin{array}{rl}
		\varphi_i(\mathbf{x}) & = \displaystyle\max_{j\in J}\, \left[ \left(\displaystyle\sum_{q\in J}\|{\mathbf{a}}^i-{\mathbf{x}}^q\|^2\right)-\|{\mathbf{a}}^i-{\mathbf{x}}^j\|^2 \right]\\
		& = \left(\displaystyle\sum_{q\in J}\|{\mathbf{a}}^i-{\mathbf{x}}^q\|^2\right) + \displaystyle\max_{j\in J}\, \left(-\|{\mathbf{a}}^i-{\mathbf{x}}^j\|^2\right)\\
		& = \left(\displaystyle\sum_{q\in J}\|{\mathbf{a}}^i-{\mathbf{x}}^q\|^2\right) - \displaystyle\min_{j\in J}\, \|{\mathbf{a}}^i-{\mathbf{x}}^j\|^2.
	\end{array}
\end{eqnarray*}
Thus, the maximum in \eqref{Phi_i function} is attained when the minimum $\displaystyle\min_{j\in J}\, \|{\mathbf{a}}^i-{\mathbf{x}}^j\|^2$ is achieved. So, by \eqref{J_i function},
\begin{eqnarray*}
	J_i(\mathbf{x}) = \left\{j\in J\,\mid\, \|{\mathbf{a}}^i-{\mathbf{x}}^j\|=\displaystyle\min_{q\in J}\, \|{\mathbf{a}}^i-{\mathbf{x}}^q\|\right\}.
\end{eqnarray*}
This implies \eqref{Jix_formula}.   $\hfill\Box$
\end{proof} 

Invoking the subdifferential formula for the maximum function (see \cite[Proposition~2.3.12]{Clarke_1990} and note that the Clarke generalized gradient coincides with the subdifferential of convex analysis if the functions in question are convex), we have 
\begin{eqnarray}\label{Subdifferential Phi_i function}
\partial\varphi_i(\mathbf{x}) = {\rm co}\left\{\nabla h_{i,j}(\mathbf{x})\, |\, j\in J_i(\mathbf{x})\right\} = {\rm co}\left\{2\left(\widetilde{\mathbf{x}}^j - \widetilde{\mathbf{a}}^{i,j}\right)\,\mid\, j\in J_i(\mathbf{x})\right\},
\end{eqnarray} 
where
\begin{eqnarray}\label{x_tilde}\widetilde{\mathbf{x}}^j = \left({\mathbf{x}}^1,\dots,{\mathbf{x}}^{j-1},\mathbf{0}_{\mathbb R^n},{\mathbf{x}}^{j+1},\dots,{\mathbf{x}}^k\right)\end{eqnarray}
and \begin{eqnarray}\label{a_tilde}\widetilde{\mathbf{a}}^{i,j} = \left({\mathbf{a}}^i,\dots,{\mathbf{a}}^i,\underbrace{\mathbf{0}_{\mathbb R^n}}_{j-{\rm th\ position}},{\mathbf{a}}^i,\dots,{\mathbf{a}}^i\right).\end{eqnarray}
By the Moreau-Rockafellar theorem \cite[Theorem~23.8]{Rockafellar_1970}, one has
\begin{eqnarray}\label{Subdifferential fk2}
\partial f^2(\mathbf{x}) = \frac{1}{m}\sum_{i\in I} \partial\varphi_i(\mathbf{x})
\end{eqnarray} with $\partial\varphi_i(\mathbf{x})$ being computed by \eqref{Subdifferential Phi_i function}. 

\medskip
Now, suppose $\mathbf{x} = ({\mathbf{x}}^1,..., {\mathbf{x}}^k)\in\mathbb R^{n\times k}$ is a local solution of \eqref{DC_clustering_problem}. By the necessary optimality condition in DC programming (see, e.g., \cite{NTVHang_Yen_JOTA2016} and \cite{PhamDinh_LeThi_AMV97}), which can be considered as a consequence of the optimality condition obtained by Dem'yanov et al. in
quasidifferential calculus (see, e.g.,  \cite[Theorem 3.1]{Dem'yanov_Rubinov_1995} and \cite[Theorem~5.1]{Dem'yanov_Vasil'ev_1985}), we have 
\begin{equation}\label{Optimality_condition_DC 1a}
\partial f^2(\mathbf{x})\subset 	\partial f^1(\mathbf{x}).
\end{equation}
Since $\partial f^1(\mathbf{x})$ is a singleton, $\partial f^2(\mathbf{x})$ must be a singleton too. This happens if and only if $\partial\varphi_i(\mathbf{x})$ is a singleton for every $i\in I$. By \eqref{Subdifferential Phi_i function}, if $|J_i(\mathbf{x})|=1$, then $|\partial\varphi_i(\mathbf{x})|=1$. In the case where $|J_i(\mathbf{x})|>1$, we can select two elements $j_1$ and $j_2$ from $J_i(\mathbf{x})$, $j_1<j_2$. As $\partial\varphi_i(\mathbf{x})$ is a singleton, by  \eqref{Subdifferential Phi_i function} one must have $\widetilde{\mathbf{x}}^{j_1} - \widetilde{\mathbf{a}}^{i,j_1}=\widetilde{\mathbf{x}}^{j_2} - \widetilde{\mathbf{a}}^{i,j_2}$. Using \eqref{x_tilde} and \eqref{a_tilde}, one sees that the latter occurs if and only if ${\mathbf{x}}^{j_1}={\mathbf{x}}^{j_2}={\mathbf{a}}^i$. To proceed furthermore, we need to introduce the following condition on the local solution $\mathbf{x}$.

\medskip
{\bf (C1)} \textit{The components of $\mathbf{x}$ are pairwise distinct, i.e., ${\mathbf{x}}^{j_1}\neq {\mathbf{x}}^{j_2}$ whenever $j_2\neq j_1$.}

\begin{definition}\label{nontrivial_local_solution}
	{\rm  A local solution $\mathbf{x} = ({\mathbf{x}}^1,..., {\mathbf{x}}^k)$ of  \eqref{DC_clustering_problem} that satisfies {\bf (C1)} is called a \textit{nontrivial local solution}.}
\end{definition}

\begin{remark}\label{nontriviality_of_global_solution} {\rm Proposition \ref{thm_basic_facts(1)} shows that every global solution of \eqref{DC_clustering_problem} is a nontrivial local solution.}
\end{remark}

The following fundamental facts have the origin in \cite[pp.~346]{Ordin_Bagirov_2015}. Here, a more precise and complete formulation is presented. In accordance with \eqref{Jix_formula}, the first assertion of next theorem means that if $\mathbf{x}$ is a nontrivial local solution, then for each data point ${\mathbf{a}}^i\in A$ there is a unique component ${\mathbf{x}}^j$ of $\mathbf{x}$ such that ${\mathbf{a}}^i\in A[{\mathbf{x}}^j]$.

\begin{theorem}\label{necessary_optimality_condition} {\rm (Necessary conditions for nontrivial local optimality)} Suppose that $\mathbf{x} = ({\mathbf{x}}^1,..., {\mathbf{x}}^k)$ is a nontrivial local solution of \eqref{DC_clustering_problem}. Then, for any $i\in I$, $|J_i(\mathbf{x})|=1$. Moreover, for every $j\in J$ such that the attraction set $A[{\mathbf{x}}^j]$ of ${\mathbf{x}}^j$ is nonempty, one has
	\begin{equation}\label{local_solution_components}
	{\mathbf{x}}^j=\frac{1}{|I(j)|}\displaystyle\sum_{i\in I(j)} {\mathbf{a}}^i,
	\end{equation} where $I(j)=\left\{i\in I\mid {\mathbf{a}}^i\in A[{\mathbf{x}}^j]\right\}$. For any $j\in J$ with $A[{\mathbf{x}}^j]=\emptyset$, one has 
	\begin{equation}\label{abnormal_components}
	{\mathbf{x}}^j\notin {\mathcal A}[\mathbf{x}],
	\end{equation} where $ {\mathcal A}[\mathbf{x}]$ is the union of the balls $\bar B({\mathbf{a}}^p,\|{\mathbf{a}}^p-{\mathbf{x}}^q\|)$ with $p\in I$, $q\in J$ satisfying $p\in I(q)$.
\end{theorem}
\begin{proof} Suppose $\mathbf{x} = ({\mathbf{x}}^1,..., {\mathbf{x}}^k)$ is a nontrivial local solution of \eqref{DC_clustering_problem}. Given any $i\in I$, we must have $|J_i(\mathbf{x})|=1$. Indeed, if  $|J_i(\mathbf{x})|>1$ then, by the analysis given before the formulation of the theorem,  there exist indexes  $j_1$ and $j_2$ from $J_i(\mathbf{x})$ such that ${\mathbf{x}}^{j_1}={\mathbf{x}}^{j_2}={\mathbf{a}}^i$. This contradicts the nontriviality of the local solution $\mathbf{x}$. Let $J_i(\mathbf{x})=\{j(i)\}$ for $i\in I$, i.e., $j(i)\in J$ is the unique element of $J_i(\mathbf{x})$.

For each $i\in I$, observe by \eqref{Phi_i function} that
$$h_{i,j}(\mathbf{x})<h_{i,j(i)}(\mathbf{x})=\varphi_i(\mathbf{x})\ \; \forall j\in J\setminus \{j(i)\}.$$
Hence, by the continuity of the functions $h_{i,j}(\mathbf{x})$, there exists an open neighborhood $U_i$ of $\mathbf{x}$ such that 
$$h_{i,j}(\mathbf{y})<h_{i,j(i)}(\mathbf{y})\ \; \forall j\in J\setminus \{j(i)\},\; \forall  \mathbf{y}\in U_i.$$ It follows that
\begin{equation}\label{function Phi_i}
\varphi_i(\mathbf{y}) = h_{i,j(i)}(\mathbf{y})\ \; \forall  \mathbf{y}\in U_i.
\end{equation}
So, $\varphi_i(.)$ is continuously differentiable on $U_i$. Put $U=\displaystyle\bigcap_{i\in I}U_i$. From \eqref{DC function fk2} and \eqref{function Phi_i} one can deduce that
\begin{eqnarray*}
	f^2(\mathbf{y})=\frac{1}{m}\sum_{i\in I} \varphi_i(\mathbf{y})=\frac{1}{m}\sum_{i\in I} h_{i,j(i)}(\mathbf{y})\quad\forall  \mathbf{y}\in U.
\end{eqnarray*} Therefore, $f^2(\mathbf{y})$ is continuously differentiable function on $U$. Moreover, the formulas \eqref{Subdifferential Phi_i function}--\eqref{a_tilde} yield
\begin{eqnarray}\label{nabla fk2}
\nabla f^2(\mathbf{y})= \frac{2}{m}\sum_{i\in I} (\widetilde {\mathbf{y}}^{j(i)}-\widetilde {\mathbf{a}}^{i,j(i)})\quad\forall  \mathbf{y}\in U,
\end{eqnarray}
where  \begin{eqnarray*}\widetilde {\mathbf{y}}^{j(i)}=\big(\mathbf{y}^1,...,\mathbf{y}^{j(i)-1},\mathbf{0}_{\mathbb{R}^n},\mathbf{y}^{j(i)+1},...,\mathbf{y}^k\big)\end{eqnarray*}
and 
\begin{eqnarray*}\widetilde {\mathbf{a}}^{i,j(i)}= \left({\mathbf{a}}^i,\dots,{\mathbf{a}}^i,\underbrace{\mathbf{0}_{\mathbb R^n}}_{j(i)-{\rm th\ position}},{\mathbf{a}}^i,\dots,{\mathbf{a}}^i\right). \end{eqnarray*} Substituting $\mathbf{y}=\mathbf{x}$ into  \eqref{nabla fk2} and combining the result with \eqref{Optimality_condition_DC 1a}, we obtain
\begin{eqnarray}\label{basic_relation}\sum_{i\in I} (\widetilde {\mathbf{x}}^{j(i)}-\widetilde {\mathbf{a}}^{i,j(i)})=m({\mathbf{x}}^1-{\mathbf{a}}^0,...,{\mathbf{x}}^k-{\mathbf{a}}^0).\end{eqnarray}
Now, fix an index $j\in J$ with $A[{\mathbf{x}}^j]\neq\emptyset$ and transform the left-hand side of \eqref{basic_relation} as follows:
\begin{align*}
\sum_{i\in I} (\widetilde {\mathbf{x}}^{j(i)}-\widetilde {\mathbf{a}}^{i,j(i)})
& = \sum_{i\in I,\, j(i)=j} (\widetilde {\mathbf{x}}^{j(i)}-\widetilde {\mathbf{a}}^{i,j(i)}) + \sum_{i\in I,\, j(i)\neq j} (\widetilde {\mathbf{x}}^{j(i)}-\widetilde {\mathbf{a}}^{i,j(i)})\\
& = \sum_{i\in I,\, j(i)=j} (\widetilde {\mathbf{x}}^{j(i)}-\widetilde {\mathbf{a}}^{i,j(i)}) + \sum_{i\notin I(j)} (\widetilde {\mathbf{x}}^{j(i)}-\widetilde {\mathbf{a}}^{i,j(i)}).
\end{align*}
Clearly, if $j(i)=j$, then the $j$-th component of the vector $\widetilde {\mathbf{x}}^{j(i)}-\widetilde {\mathbf{a}}^{i,j(i)}$, that belongs to $\mathbb R^{n\times k}$, is $\mathbf{0}_{\mathbb R^n}$. If $j(i)\neq j$, then the $j$-th component of the vector $\widetilde {\mathbf{x}}^{j(i)}-\widetilde {\mathbf{a}}^{i,j(i)}$ is ${\mathbf{x}}^j-{\mathbf{a}}^i$. Consequently, \eqref{basic_relation} gives us
$$\displaystyle\sum_{i\notin I(j)}({\mathbf{x}}^j-{\mathbf{a}}^i)=m({\mathbf{x}}^j-{\mathbf{a}}^0).$$ Since $m{\mathbf{a}}^0={\mathbf{a}}^1+\dots +{\mathbf{a}}^m$, this yields $\displaystyle\sum_{i\in I(j)}{\mathbf{a}}^i = |I(j)|{\mathbf{x}}^j.$ Thus, formula \eqref{local_solution_components} is valid for any $j\in J$ satisfying $A[{\mathbf{x}}^j]\neq\emptyset$.

For any $j\in J$ with $A[{\mathbf{x}}^j]=\emptyset$, one has \eqref{abnormal_components}. Indeed, suppose to the contrary that there exits $j_0\in J$ with $A[{\mathbf{x}}^{j_0}]=\emptyset$ such that for some $p\in I$, $q\in J$, one has $p\in I(q)$ and 
${\mathbf{x}}^{j_0}\in\bar B({\mathbf{a}}^p,\|{\mathbf{a}}^p-{\mathbf{x}}^q\|)$. If $\|{\mathbf{a}}^p-{\mathbf{x}}^{j_0}\|=\|{\mathbf{a}}^p-{\mathbf{x}}^q\|$, then $J_p(\mathbf{x})\supset \{q,j_0\}$. This is impossible due to the first claim of the theorem. Now, if $\|{\mathbf{a}}^p-{\mathbf{x}}^{j_0}\|<\|{\mathbf{a}}^p-{\mathbf{x}}^q\|$, then $p\notin I(q)$. We have thus arrived at a contradiction.

The proof is complete. $\hfill\Box$
\end{proof}

Roughly speaking, the necessary optimality condition given in the above theorem is a sufficient one. Therefore, in combination with Theorem \ref{necessary_optimality_condition}, next statement gives \textit{a complete description} of the nontrivial local solutions of \eqref{DC_clustering_problem}.

\begin{theorem}\label{sufficient_optimality_condition}  {\rm (Sufficient conditions for nontrivial local optimality)}  Suppose that a vector $\mathbf{x} = ({\mathbf{x}}^1,..., {\mathbf{x}}^k) \in\mathbb R^{n\times k}$ satisfies condition  {\bf (C1)} and $|J_i(\mathbf{x})|=1$ for every $i\in I$. If \eqref{local_solution_components} is valid for any $j\in J$ with $A[{\mathbf{x}}^j]\neq\emptyset$ and \eqref{abnormal_components} is fulfilled for any $j\in J$ with $A[{\mathbf{x}}^j]=\emptyset$,  then $\mathbf{x}$ is a nontrivial local solution of~\eqref{DC_clustering_problem}.
\end{theorem}
\begin{proof} Let $\mathbf{x} = ({\mathbf{x}}^1,..., {\mathbf{x}}^k) \in\mathbb R^{n\times k}$ be such that  {\bf (C1)} holds, $J_i(\mathbf{x})=\{j(i)\}$ for every $i\in I$, \eqref{local_solution_components} is valid for any $j\in J$ with $A[{\mathbf{x}}^j]\neq\emptyset$, and \eqref{abnormal_components} is satisfied for any $j\in J$ with $A[{\mathbf{x}}^j]=\emptyset$. Then, for all $i\in I$ and $j'\in J\setminus \{j(i)\}$, one has
$$\|{\mathbf{a}}^i-{\mathbf{x}}^{j(i)}\|<\|{\mathbf{a}}^i-{\mathbf{x}}^{j'}\|.$$ So, there exist $\varepsilon>0$, $q\in J$, such that 
\begin{eqnarray}\label{perturbation1}
\|{\mathbf{a}}^i-\widetilde{\mathbf{x}}^{j(i)}\|<\|{\mathbf{a}}^i-\widetilde{\mathbf{x}}^{j'}\|\quad \forall i\in I,\; \forall j'\in J\setminus \{j(i)\},
\end{eqnarray} whenever vector $\widetilde{\mathbf{x}} = (\widetilde{\mathbf{x}}^1,..., \widetilde{\mathbf{x}}^k) \in\mathbb R^{n\times k}$ satisfies the condition $\|\widetilde{\mathbf{x}}^q-{\mathbf{x}}^q\|<\varepsilon$ for all $q\in J$.
By \eqref{abnormal_components} and by the compactness of ${\mathcal A}[\mathbf{x}]$, reducing the positive number $\varepsilon$ (if necessary) we have \begin{equation}\label{abnormal_components_perturbed}
\widetilde{\mathbf{x}}^j\notin {\mathcal A}[\widetilde{\mathbf{x}}]
\end{equation} whenever vector $\widetilde{\mathbf{x}} = (\widetilde{\mathbf{x}}^1,..., \widetilde{\mathbf{x}}^k) \in\mathbb R^{n\times k}$ satisfies the condition $\|\widetilde{\mathbf{x}}^q-{\mathbf{x}}^q\|<\varepsilon$ for all $q\in J$, where $ {\mathcal A}[\widetilde{\mathbf{x}}]$ is the union of the balls $\bar B({\mathbf{a}}^p,\|{\mathbf{a}}^p-\widetilde{\mathbf{x}}^q\|)$ with $p\in I$, $q\in J$ satisfying $p\in I(q)=\left\{i\in I\mid {\mathbf{a}}^i\in A[{\mathbf{x}}^q]\right\}$.

Fix an arbitrary vector $\widetilde{\mathbf{x}} = (\widetilde{\mathbf{x}}^1,..., \widetilde{\mathbf{x}}^k) \in\mathbb R^{n\times k}$ with the property that $\|\widetilde{\mathbf{x}}^q-{\mathbf{x}}^q\|<\varepsilon$ for all $q\in J$. Then, by \eqref{perturbation1} and \eqref{abnormal_components_perturbed}, $J_i(\widetilde{\mathbf{x}})=\{j(i)\}$. So, 
$$\min_{j\in J} \|{\mathbf{a}}^i-\widetilde{\mathbf{x}}^j\|^2=\|{\mathbf{a}}^i-\widetilde{\mathbf{x}}^{j(i)}\|^2.$$ Therefore, one has
\begin{equation*}\begin{array}{rl}
f(\widetilde{\mathbf{x}}) & =\displaystyle\frac{1}{m}\sum_{i\in I}\left(\min_{j\in J} \|{\mathbf{a}}^i-\widetilde{\mathbf{x}}^j\|^2\right)\\
& = \displaystyle\frac{1}{m}\sum_{i\in I}\|{\mathbf{a}}^i-\widetilde{\mathbf{x}}^{j(i)}\|^2\\
& = \displaystyle\frac{1}{m}\sum_{j\in J}\left(\sum_{i\in I(j)}\|{\mathbf{a}}^i-\widetilde{\mathbf{x}}^{j(i)}\|^2\right)\\
& = \displaystyle\frac{1}{m}\sum_{j\in J}\left(\sum_{i\in I(j)}\|{\mathbf{a}}^i-\widetilde{\mathbf{x}}^{j}\|^2\right)\\
& \geq \displaystyle\frac{1}{m}\sum_{j\in J}\left(\sum_{i\in I(j)}\|{\mathbf{a}}^i-{\mathbf{x}}^{j}\|^2\right)=f(\mathbf{x}),
\end{array}
\end{equation*}
where the inequality is valid because \eqref{local_solution_components} obviously yields
$$\sum_{i\in I(j)}\|{\mathbf{a}}^i-{\mathbf{x}}^{j}\|^2\leq \sum_{i\in I(j)}\|{\mathbf{a}}^i-\widetilde{\mathbf{x}}^{j}\|^2$$ for every $j\in J$ such that the attraction set $A[{\mathbf{x}}^j]$ of ${\mathbf{x}}^j$ is nonempty. (Note that ${\mathbf{x}}^j$ is the barycenter of  $A[{\mathbf{x}}^j]$.)

The local optimality of $\mathbf{x}= ({\mathbf{x}}^1,..., {\mathbf{x}}^k) $ has been proved. Hence, $\mathbf{x}$ is a nontrivial local solution of \eqref{DC_clustering_problem}.$\hfill\Box$
\end{proof}

\begin{example}\label{basic_example} {\rm (A local solution need not be a global solution)} {\rm Consider the clustering problem described in Example \ref{Kmeans_example}. Here, $I=\{1,2,3\}$ and $J=\{1,2\}$. By Theorem \ref{thm_basic_facts}, problem \eqref{DC_clustering_problem} has a global solution. Moreover, if $\mathbf{x}=({\mathbf{x}}^1,{\mathbf{x}}^2)\in\mathbb R^{2\times 2}$ is a global solution then, for every $j\in J$, the attraction set $A[{\mathbf{x}}^j]$ is nonempty. Thanks to Remark \ref{nontriviality_of_global_solution}, we know that $\mathbf{x}$ is a nontrivial local solution. So, by Theorem \ref{necessary_optimality_condition}, the attraction sets $A[{\mathbf{x}}^1]$ and $A[{\mathbf{x}}^2]$ are disjoint. Moreover, the barycenter of each one of these sets can be computed by formula \eqref{local_solution_components}. Clearly, $A=A[{\mathbf{x}}^1]\cup A[{\mathbf{x}}^2]$. Since $A[{\mathbf{x}}^j]\subset A=\{{\mathbf{a}}^1,{\mathbf{a}}^2,{\mathbf{a}}^3\}$, allowing permutations of the components of each vector $\mathbf{x}=({\mathbf{x}}^1,{\mathbf{x}}^2)\in\mathbb R^{2\times 2}$ (see Remark \ref{permutations_are_allowed}), we can assert that the global solution set of our problem is contained in the set
		\begin{equation}\label{core_local_solutions}
		\left\{\bar{\mathbf{x}}:=\big((\frac{1}{2},\frac{1}{2}),\,(0,0)\big),\ \hat {\mathbf{x}}:=\big((0,\frac{1}{2}),\,(1,0)\big),\ \widetilde {\mathbf{x}}:=\big((\frac{1}{2},0),(0,1)\big)\right\}.
		\end{equation}
		Since $f(\bar{\mathbf{x}}) = \frac{1}{3}$ and $f(\hat {\mathbf{x}}) = f(\widetilde {\mathbf{x}}) = \frac{1}{6}$, we infer that $\hat {\mathbf{x}}$ and $\widetilde {\mathbf{x}}$ are global solutions of our problem. Using Theorem \ref{sufficient_optimality_condition}, we can assert that $\bar{\mathbf{x}}$ is a local solution. Thus,  $\bar{\mathbf{x}}$ is a local solution which does not belong to the global solution set, i.e., $\bar{\mathbf{x}}$ is \textit{a local-nonglobal solution} of our problem.}
\end{example}

\begin{example}\label{local_solutions_example} {\rm (Complete description of the set of nontrivial local solutions)} {\rm Again, consider the MSSC problem given in Example \ref{Kmeans_example}. Allowing permutations of the components of each vector in $\mathbb R^{2\times 2}$, by Theorems \ref{necessary_optimality_condition} and \ref{sufficient_optimality_condition} we find that the set of nontrivial local solutions consists of the three vectors described in \eqref{core_local_solutions} and all the vectors of the form $\mathbf{x}=({\mathbf{x}}^1,{\mathbf{x}}^2)\in\mathbb R^{2\times 2}$, where ${\mathbf{x}}^1 = (\frac{1}{3},\frac{1}{3})$ and $${\mathbf{x}}^2\notin \bar B({\mathbf{a}}^1,\|{\mathbf{a}}^1-{\mathbf{x}}^1\|)\cup \bar B({\mathbf{a}}^2,\|{\mathbf{a}}^2-{\mathbf{x}}^1\|)\cup \bar B({\mathbf{a}}^3,\|{\mathbf{a}}^3-{\mathbf{x}}^1\|).$$ This set of nontrivial local solutions is \textit{unbounded} and \textit{non-closed}.}
\end{example}

\begin{example}\label{convergence_analysis_example} {\rm (Convergence analysis of the $k$-means algorithm)} {\rm 
		Consider once again the problem described in Example~\ref{Kmeans_example}. By the results given in Example~\ref{local_solutions_example}, the centroid systems in items (a), (b), (c) and (d) of Example~\ref{Kmeans_example} are local solutions. In addition, by Example \ref{basic_example}, the centroid systems in the just mentioned items (a) and (c) are global solutions. Concerning the centroid system in item (e) of Example~\ref{Kmeans_example}, remark that $\mathbf{x}:=\big((\frac{1}{3},\frac{1}{3}), (1+\frac{\sqrt{5}}{3},0)\big)$ is not a local solution by Theorem \ref{necessary_optimality_condition}, because ${\mathbf{a}}^2\in A[{\mathbf{x}}^1]\cap A[{\mathbf{x}}^2]$, i.e., $J_2(\mathbf{x}) = \{1,2\}$ (see \figurename{ 1}). In general, with ${\mathbf{x}}^1=(\frac{1}{3},\frac{1}{3})$ and ${\mathbf{x}}^2\in\mathbb{R}^{2\times 2}$ belonging to the boundary of the set $$\bar B({\mathbf{a}}^1,\|{\mathbf{a}}^1-{\mathbf{x}}^1\|)\cup \bar B({\mathbf{a}}^2,\|{\mathbf{a}}^2-{\mathbf{x}}^1\|)\cup \bar B({\mathbf{a}}^3,\|{\mathbf{a}}^3-{\mathbf{x}}^1\|),$$
		$\mathbf{x}:=({\mathbf{x}}^1,{\mathbf{x}}^2)$ is not a local solution of the MSSC problem under consideration. The above analysis shows that the $k$-means algorithm is very sensitive to the choice of starting centroids. The algorithm may give a global solution, a local-nonglobal solution, as well as a centroid system which is not a local solution. In other words, the quality of the obtained result greatly depends on the initial centroid system.} 
\end{example}

	\begin{figure}[htp]
		\begin{center}
			\includegraphics[height=8.5cm,width=8cm]{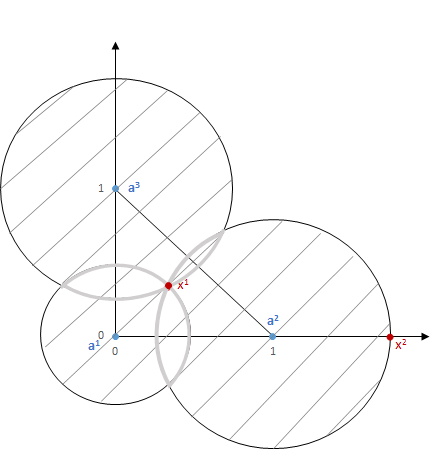}
		\end{center}
		\caption{The centroids in item (e) of Example \ref{Kmeans_example}}\label{fig 1}
	\end{figure}

\section{Stability Properties}

This section is devoted to establishing the local Lipschitz property of the optimal value function, the local upper Lipschitz property of the global solution map, and the local Lipschitz-like property of the local solution map of \eqref{DC_clustering_problem}.

\medskip
Now, let the data set $A = \{{\mathbf{a}}^1,...,{\mathbf{a}}^m\}$ of the problem \eqref{DC_clustering_problem} be subject to change. Put $\mathbf{a}=({\mathbf{a}}^1,...,{\mathbf{a}}^m)$ and observe that $\mathbf{a}\in \mathbb R^{n\times m}$. Denoting by $v(\mathbf{a})$ the \textit{optimal value} of \eqref{DC_clustering_problem}, one has \begin{equation}\label{formula_v(a)}
v(\mathbf{a})=\min\{f(\mathbf{x})\,\mid\, \mathbf{x}=({\mathbf{x}}^1,\dots,{\mathbf{x}}^k)\in\mathbb R^{n\times k}\}.
\end{equation} The\textit{ global solution set} of \eqref{DC_clustering_problem}, denoted by $F(\mathbf{a})$, is given by $$F(\mathbf{a})=\left\{\mathbf{x}=({\mathbf{x}}^1,\dots,{\mathbf{x}}^k)\in\mathbb R^{n\times k} \,\mid\, f(\mathbf{x})=v(\mathbf{a})\right\}.$$ Let us abbreviate the \textit{local solution set} of \eqref{DC_clustering_problem} to $F_1(\mathbf{a})$. Note that $F(\mathbf{a})\subset F_1(\mathbf{a})$, and the inclusion may be strict. 

\begin{definition}\label{partition_of_I}
	{\rm A family $\{I(j)\mid j\in J\}$ of pairwise distinct, nonempty subsets of $I$ is said to be a \textit{partition} of $I$ if $\displaystyle\bigcup_{j\in J} I(j)=I$.}
\end{definition}

\medskip
\textit{From now on, let $\bar {\mathbf{a}}=(\bar {\mathbf{a}}^1,...,\bar {\mathbf{a}}^m)\in \mathbb R^{n\times m}$ be a fixed vector with the property that $\bar {\mathbf{a}}^1,...,\bar {\mathbf{a}}^m$ are pairwise distinct.}

\begin{theorem}\label{Lipschitz_property_of_v}  {\rm (Local Lipschitz property of the optimal value function)}  The optimal value function $v:\mathbb R^{n\times m}\to\mathbb R$ is locally Lipschitz at $\bar {\mathbf{a}}$, i.e., there exist $L_0>0$ and $\delta_0>0$ such that $$|v(\mathbf{a})-v(\mathbf{a}')|\leq L_0 \|\mathbf{a}-\mathbf{a}'\|$$ for all $\mathbf{a}$ and $\mathbf{a'}$ satisfying $\|\mathbf{a}-\bar{\mathbf{a}}\|<\delta_0$ and $\|\mathbf{a}'-\bar{\mathbf{a}}\|<\delta_0$.
\end{theorem}
\begin{proof} Denote by $\Omega$ the set of all the partitions of $I$. Every element $\omega$ of $\Omega$ is a family $\{I_\omega(j)\mid j\in J\}$ of pairwise distinct, nonempty subsets of $I$ with $\displaystyle\bigcup_{j\in J} I_\omega(j)=I$. We associate to each pair $(\omega,\mathbf{a})$, where $\mathbf{a}=({\mathbf{a}}^1,...,{\mathbf{a}}^m)\in \mathbb R^{n\times m}$ and $\omega\in\Omega$, a vector $\mathbf{x}_\omega(\mathbf{a})=({\mathbf{x}}^1_\omega(\mathbf{a}),\dots,{\mathbf{x}}^k_\omega(\mathbf{a}))\in\mathbb R^{n\times k}$ with 
\begin{equation}\label{x_omega} {\mathbf{x}}^j_\omega(\mathbf{a})=\frac{1}{|I_\omega(j)|}\displaystyle\sum_{i\in I_\omega(j)} {\mathbf{a}}^i\end{equation} for every $j\in J$.
By Theorem \ref{thm_basic_facts}, problem \eqref{DC_clustering_problem} has solutions and the number of the global solutions is finite, i.e., $F(\bar {\mathbf{a}})$ is nonempty and finite. Moreover, for each $\bar{\mathbf{x}}=(\bar {\mathbf{x}}^1,...,\bar {\mathbf{x}}^k)\in F(\bar {\mathbf{a}})$, one can find some $\omega\in\Omega$ satisfying $\bar {\mathbf{x}}^j={\mathbf{x}}^j_\omega(\bar {\mathbf{a}})$ for all $j\in J$. Let $\Omega_1=\left\{\omega_1,\dots,\omega_r\right\}$ be the set of the elements of $\Omega$ corresponding the global solutions. Then, 
\begin{equation}\label{strict_ineq}
f(\mathbf{x}_{\omega_1}(\bar {\mathbf{a}}),\bar {\mathbf{a}})< f(\mathbf{x}_\omega(\bar {\mathbf{a}}),\bar {\mathbf{a}})\quad\ (\forall \omega\in\Omega\setminus\Omega_1),
\end{equation}
where
\begin{equation}\label{objective_function_at_a}
f(\mathbf{x},\mathbf{a})=\frac{1}{m}\sum_{i\in I}\left(\min_{j\in J} \|{\mathbf{a}}^i-{\mathbf{x}}^j\|^2\right).
\end{equation}
\hskip 0.5cm For each pair $(i,j)\in I\times J$, the rule $(\mathbf{x},\mathbf{a})\mapsto \|{\mathbf{a}}^i-{\mathbf{x}}^j\|^2$  defines a polynomial function on $\mathbb R^{n\times k}\times\mathbb R^{n\times m}$. In particular, this function is locally Lipschitz on its domain. So, by \cite[Prop.~2.3.6 and 2.3.12]{Clarke_1990} we can assert that the function $f(\mathbf{x},\mathbf{a})$ in \eqref{objective_function_at_a} is locally Lipschitz on $\mathbb R^{n\times k}\times\mathbb R^{n\times m}$.

Now, observe that for any $\omega\in\Omega$ and $j\in J$, the vector function ${\mathbf{x}}^j_\omega(.)$ in \eqref{x_omega}, which maps $\mathbb R^{n\times m}$ to $\mathbb R^n$, is continuously differentiable. In particular, it is locally Lipschitz on $\mathbb R^{n\times m}$.

For every $\omega\in\Omega$, from the above observations we can deduce that the function $g_\omega(\mathbf{a}):=f(\mathbf{x}_\omega(\mathbf{a}),\mathbf{a})$ is locally Lipschitz on $\mathbb R^{n\times m}$. Rewriting \eqref{strict_ineq} as
$$g_{\omega_1}(\bar {\mathbf{a}})< g_\omega(\bar {\mathbf{a}})\quad\ (\forall \omega\in\Omega\setminus\Omega_1)$$ and using the continuity of the functions $g_\omega(.)$, we can find a number $\delta_0>0$ such that  
\begin{equation}\label{strict_inequality_a}g_{\omega_1}(\mathbf{a})< g_\omega(\mathbf{a})\quad\ (\forall \omega\in\Omega\setminus\Omega_1)\end{equation}  
for all $\mathbf{a}$ satisfying $\|\mathbf{a}-\bar{\mathbf{a}}\|<\delta_0$. Since $\bar {\mathbf{a}}^1,...,\bar {\mathbf{a}}^m$ are pairwise distinct, without loss of generality, we may assume that ${\mathbf{a}}^1,...,{\mathbf{a}}^m$ are pairwise distinct for any $\mathbf{a}=({\mathbf{a}}^1,...,{\mathbf{a}}^m)$ with $\|\mathbf{a}-\bar{\mathbf{a}}\|<\delta_0$. 

Now, consider a vector $\mathbf{a}=({\mathbf{a}}^1,...,{\mathbf{a}}^m)$ satisfying $\|\mathbf{a}-\bar{\mathbf{a}}\|<\delta_0$. By \eqref{strict_inequality_a}, $f(\mathbf{x}_{\omega_1}(\mathbf{a}),\mathbf{a})<f(\mathbf{x}_\omega(\mathbf{a}),\mathbf{a})$ for all $\omega\in\Omega\setminus\Omega_1$. Since $f(.,\mathbf{a})$ is the objective function of \eqref{DC_clustering_problem}, this implies that the set $\{x_\omega(\mathbf{a})\mid \omega\in\Omega\setminus\Omega_1\}$ does not contain any global solution of the problem. Thanks to Theorem \ref{thm_basic_facts}, we know that the global solution set $F(\mathbf{a})$ of \eqref{DC_clustering_problem} is contained in the set $\{x_\omega(\mathbf{a})\mid \omega\in\Omega_1\}$. Hence,
\begin{equation}\label{basic_inclusion} F(\mathbf{a})\subset \{x_\omega(\mathbf{a})\mid \omega\in\Omega_1\}=\{x_{\omega_1}(\mathbf{a}),\dots,x_{\omega_r}(\mathbf{a})\}.\end{equation} 
Since $F(\mathbf{a})\neq\emptyset$, by \eqref{basic_inclusion} one has \begin{equation*}v(\mathbf{a})=\min\,\{f(\mathbf{x},\mathbf{a})\mid \mathbf{x}\in F(\mathbf{a})\}=\min\,\{f(\mathbf{x}_{\omega_\ell}(\mathbf{a}),\mathbf{a})\mid \ell=1,\dots,r\}.\end{equation*} Thus, we have proved that
\begin{equation}\label{value_function_a}
v(\mathbf{a})=\min\,\{g_{\omega_\ell}(\mathbf{a})\mid \ell=1,\dots,r\}
\end{equation} for all $\mathbf{a}$ satisfying $\|\mathbf{a}-\bar{\mathbf{a}}\|<\delta_0$. As it has been noted, the functions $g_\omega$, $\omega\in\Omega$, are locally Lipschitz on $\mathbb R^{n\times m}$. Hence,  applying \cite[Prop.~2.3.6 and 2.3.12]{Clarke_1990} to the minimum function in \eqref{value_function_a}, we can assert that $v$ is locally Lipschitz at $\bar{\mathbf{a}}$.

The proof is complete.  $\hfill\Box$
\end{proof}

\begin{theorem}\label{upper_Lipschitz_property_of_v}  {\rm (Local upper Lipschitz property of the global solution map)}  The global solution map $F:\mathbb R^{n\times m}\rightrightarrows\mathbb R^{n\times k}$ is locally upper Lipschitz at $\bar {\mathbf{a}}$, i.e., there exist $L>0$ and $\delta>0$ such that \begin{equation}\label{upper_Lipschitz} F({\mathbf{a}})\subset F(\bar {\mathbf{a}})+L \|\mathbf{a}-\bar{\mathbf{a}}\|\bar B_{\mathbb R^{n\times k}}\end{equation} for all $\mathbf{a}$ satisfying $\|\mathbf{a}-\bar{\mathbf{a}}\|<\delta$.  Here $$\bar B_{\mathbb R^{n\times k}}:=\left\{\mathbf{x}=({\mathbf{x}}^1,\dots,{\mathbf{x}}^k)\in\mathbb R^{n\times k}\; \mid\; \displaystyle\sum_{j\in J}\|{\mathbf{x}}^j\|\leq 1\right\}$$ denotes the closed unit ball of the product space $\mathbb R^{n\times k}$, which is equipped with the sum norm $\|\mathbf{x}\|=\displaystyle\sum_{j\in J}\|{\mathbf{x}}^j\|$.
\end{theorem}
\begin{proof} Let $\Omega$, $\Omega_1=\left\{\omega_1,\dots,\omega_r\right\}$, $\mathbf{x}_\omega(\mathbf{a})=({\mathbf{x}}^1_\omega(\mathbf{a}),\dots,{\mathbf{x}}^k_\omega(\mathbf{a}))\in\mathbb R^{n\times k}$, and $\delta_0$ be constructed as in the proof of the above theorem.  For any $\omega\in\Omega$, the vector function $\mathbf{x}_\omega(.)$, which maps $\mathbb R^{n\times m}$ to $\mathbb R^{n\times k}$, is continuously differentiable. Hence, there exist $L_\omega>0$ and $\delta_\omega>0$ such that
\begin{equation}\label{x_omega_ell}
\|\mathbf{x}_\omega(\mathbf{a})-x_\omega(\widetilde {\mathbf{a}})\|\leq L_\omega\|\mathbf{a}-\widetilde {\mathbf{a}}\|
\end{equation} for any $\mathbf{a},\, \widetilde {\mathbf{a}}$  satisfying $\|\mathbf{a}-\bar{\mathbf{a}}\|<\delta_\omega$ and $\|\widetilde {\mathbf{a}}-\bar{\mathbf{a}}\|<\delta_\omega$. Set $$L=\max\{L_{\omega_1},\dots,L_{\omega_r}\}\ \; {\rm and}\ \; \delta=\min\{\delta_0,\delta_{\omega_1}\dots,\delta_{\omega_r}\}.$$ Then, for every $\mathbf{a}$ satisfying $\|\mathbf{a}-\bar{\mathbf{a}}\|<\delta$, by \eqref{basic_inclusion} and \eqref{x_omega_ell} one has 
\begin{eqnarray*}\begin{array}{rcl}
		F(\mathbf{a})\subset \{x_{\omega_1}(\mathbf{a}),\dots,x_{\omega_r}(\mathbf{a})\} & \subset& \{x_{\omega_1}(\bar {\mathbf{a}}),\dots,x_{\omega_r}(\bar {\mathbf{a}})\}+L \|\mathbf{a}-\bar{\mathbf{a}}\|\bar B_{\mathbb R^{n\times k}}\\
		& =& F(\bar {\mathbf{a}})+L \|\mathbf{a}-\bar{\mathbf{a}}\|\bar B_{\mathbb R^{n\times k}}.
	\end{array}
\end{eqnarray*}
Hence, inclusion \eqref {upper_Lipschitz} is valid for every $\mathbf{a}$  satisfying $\|\mathbf{a}-\bar{\mathbf{a}}\|<\delta$.
$\hfill\Box$
\end{proof}

\begin{theorem}\label{Aubin property_of_F_1}  {\rm (Aubin property of the local solution map)}  Let  $\bar{\mathbf{x}}=(\bar {\mathbf{x}}^1,...,\bar {\mathbf{x}}^k)$ be an element of $F_1(\bar {\mathbf{a}})$ satisfying condition {\bf (C1)}, that is, $\bar {\mathbf{x}}^{j_1}\neq\bar {\mathbf{x}}^{j_2}$ whenever $j_2\neq j_1$. Then, the local solution map $F_1:\mathbb R^{n\times m}\rightrightarrows\mathbb R^{n\times k}$ has the Aubin property at $(\bar {\mathbf{a}},\bar{\mathbf{x}})$, i.e., there exist $L_1>0$, $\varepsilon>0$, and $\delta_1>0$ such that \begin{equation}\label{Aubin property} F_1(\mathbf{a})\cap B(\bar{\mathbf{x}},\varepsilon)\subset F_1(\widetilde {\mathbf{a}})+L_1 \|\mathbf{a}-\widetilde{\mathbf{a}}\|\bar B_{\mathbb R^{n\times k}}\end{equation}  for all $\mathbf{a}$ and $\widetilde {\mathbf{a}}$ satisfying $\|\mathbf{a}-\bar{\mathbf{a}}\|<\delta_1$ and $\|\widetilde {\mathbf{a}}-\bar{\mathbf{a}}\|<\delta_1$.
\end{theorem}
\begin{proof}  Suppose that $\bar{\mathbf{x}}=(\bar {\mathbf{x}}^1,...,\bar {\mathbf{x}}^k)\in F_1(\bar {\mathbf{a}})$ and  $\bar {\mathbf{x}}^{j_1}\neq\bar {\mathbf{x}}^{j_2}$ for all $j_1,j_2\in J$ with $j_2\neq j_1$. Denote by $J_1$ the set of the indexes $j\in J$ such that $\bar {\mathbf{x}}^j$ is attractive w.r.t. the data set $\{\bar {\mathbf{a}}^1,\dots,\bar {\mathbf{a}}^m\}$. Put $J_2=J\setminus J_1$. For every $j\in J_1$, by Theorem 
\ref{necessary_optimality_condition} one has
\begin{equation}\label{strict_inequalities_1}
\|\bar {\mathbf{a}}^i-\bar {\mathbf{x}}^j\|<\|\bar {\mathbf{a}}^i-\bar {\mathbf{x}}^q\|\quad\; (\forall i\in I(j),\; \forall q\in J\setminus\{j\}).\end{equation}
In addition, the following holds:
\begin{equation}\label{local_solution_components_1}
\bar {\mathbf{x}}^j=\frac{1}{|I(j)|}\displaystyle\sum_{i\in I(j)} \bar {\mathbf{a}}^i,\end{equation} where $I(j)=\left\{i\in I\mid \bar {\mathbf{a}}^i\in A[\bar {\mathbf{x}}^j]\right\}$.  For every $j\in J_2$, by Theorem \ref{necessary_optimality_condition} one has
\begin{equation}\label{strict_inequalities_2}
\|\bar {\mathbf{x}}^q-\bar {\mathbf{a}}^p\|<\|\bar {\mathbf{x}}^j-\bar {\mathbf{a}}^p\|\quad\; (\forall q\in J_1,\; \forall p\in I(q)).\end{equation} Let $\varepsilon_0>0$ be such that $\|\bar {\mathbf{x}}^{j_1}-\bar {\mathbf{x}}^{j_2}\|>\varepsilon_0$ for all $j_1,j_2\in J$ with $j_2\neq j_1$.

By \eqref{strict_inequalities_1} and \eqref{strict_inequalities_2}, there exist $\delta_0> 0$ and $\varepsilon \in \left(0,\displaystyle\frac{\varepsilon_0}{4}\right)$ such that 
\begin{equation}\label{strict_inequalities_1a}
\|{\mathbf{a}}^i-{\mathbf{x}}^j\|<\|{\mathbf{a}}^i-{\mathbf{x}}^q\|\quad\; (\forall j\in J_1,\; \forall i\in I(j),\; \forall q\in J\setminus\{j\})\end{equation}
and
\begin{equation}\label{strict_inequalities_2a}
\|{\mathbf{x}}^q-{\mathbf{a}}^p\|<\|{\mathbf{x}}^j-{\mathbf{a}}^p\|\quad\; (\forall j\in J_2,\; \forall q\in J_1,\; \forall p\in I(q))\end{equation}
for all $\mathbf{a}=({\mathbf{a}}^1,...,{\mathbf{a}}^m)\in\mathbb R^{n\times m}$ and $\mathbf{x}=({\mathbf{x}}^1,\dots,{\mathbf{x}}^k)\in\mathbb R^{n\times k}$ satisfying $\|\mathbf{a}-\bar{\mathbf{a}}\|<\delta_0$ and $\|\mathbf{x}-\bar{\mathbf{x}}\|<2k\varepsilon$. As  $\bar {\mathbf{x}}^{j_1}\neq\bar {\mathbf{x}}^{j_2}$ for all $j_1,j_2\in J$ with $j_2\neq j_1$, by taking a smaller $\varepsilon>0$ (if necessary), for any  $\mathbf{x}=({\mathbf{x}}^1,\dots,{\mathbf{x}}^k)\in\mathbb R^{n\times k}$ satisfying  $\|\mathbf{x}-\bar{\mathbf{x}}\|<2k\varepsilon$ we have ${\mathbf{x}}^{j_1}\neq {\mathbf{x}}^{j_2}$ for all $j_1,j_2\in J$ with $j_2\neq j_1$.

For every $j\in J_1$ and $\mathbf{a}=({\mathbf{a}}^1,...,{\mathbf{a}}^m)\in\mathbb R^{n\times m}$, define
\begin{equation}\label{local_solution_components_1a}
{\mathbf{x}}^j(\mathbf{a})=\frac{1}{|I(j)|}\displaystyle\sum_{i\in I(j)} {\mathbf{a}}^i.
\end{equation} Comparing \eqref{local_solution_components_1a} with \eqref{local_solution_components_1} yields ${\mathbf{x}}^j(\bar {\mathbf{a}})=\bar {\mathbf{x}}^j$ for all $j\in J_1$. Then, by the continuity of the vector functions ${\mathbf{x}}^j(.)$, where $j\in J_1$, we may assume that $\|{\mathbf{x}}^j(\widetilde  {\mathbf{a}})-\bar {\mathbf{x}}^j\|<\varepsilon$ for all $j\in J_1$ and $\widetilde  a=(\widetilde {\mathbf{a}}^1,...,\widetilde {\mathbf{a}}^m)\in\mathbb R^{n\times m}$ satisfying $\|\widetilde {\mathbf{a}}-\bar{\mathbf{a}}\|<\delta_0$ (one can take a smaller $\delta_0>0$, if necessary).

Since the vector functions ${\mathbf{x}}^j(.)$, $j\in J_1$, are continuously differentiable, there exist $L_1>0$ such that 
\begin{equation}\label{x_j_a}
\|{\mathbf{x}}^j(\mathbf{a})-{\mathbf{x}}^j(\widetilde {\mathbf{a}})\|\leq \frac{1}{k} L_1\|\mathbf{a}-\widetilde {\mathbf{a}}\|
\end{equation} for any $\mathbf{a},\, \widetilde {\mathbf{a}}$  satisfying $\|\mathbf{a}-\bar{\mathbf{a}}\|<\delta_0$ and $\|\widetilde {\mathbf{a}}-\bar{\mathbf{a}}\|<\delta_0$ (one can take a smaller $\delta_0>0$, if necessary).
Choose $\delta_1\in (0,\delta_0)$ as small as $\frac{2}{k} L_1\delta_1<\varepsilon$.

With the chosen constants $L_1>0$, $\varepsilon>0$, and $\delta_1>0$,  let us show that the inclusion \eqref{Aubin property} is fulfilled  for all $\mathbf{a}$ and $\widetilde {\mathbf{a}}$ satisfying $\|\mathbf{a}-\bar{\mathbf{a}}\|<\delta_1$ and $\|\widetilde {\mathbf{a}}-\bar {\mathbf{a}} \|<\delta_1$.

Let $\mathbf{a}$ and $\widetilde  {\mathbf{a}}$ be such that $\|\mathbf{a}-\bar{\mathbf{a}}\|<\delta_1$ and $\|\widetilde {\mathbf{a}}-\bar {\mathbf{a}}\|<\delta_1$. Select an arbitrary element $\mathbf{x}=({\mathbf{x}}^1,\dots,{\mathbf{x}}^k)$ of the set $F_1(\mathbf{a})\cap B(\bar{\mathbf{x}},\varepsilon)$. Put $\widetilde {\mathbf{x}}^j={\mathbf{x}}^j(\widetilde {\mathbf{a}})$ for all $j\in J_1$, where ${\mathbf{x}}^j(\mathbf{a})$ is given by \eqref{local_solution_components_1a}. For any $j\in J_2$, set $\widetilde {\mathbf{x}}^j={\mathbf{x}}^j$.

{\sc Claim 1.} \textit{The vector $\widetilde {\mathbf{x}}=(\widetilde {\mathbf{x}}^1,\dots,\widetilde {\mathbf{x}}^k)$ belongs to $F_1(\widetilde {\mathbf{a}})$.}

Indeed, the inequalities $\|\mathbf{a}-\bar{\mathbf{a}}\|<\delta_1$ and  $\|\mathbf{x}-\bar{\mathbf{x}}\|<\varepsilon$ imply that both properties \eqref{strict_inequalities_1a} and \eqref{strict_inequalities_2a}
are available. From \eqref{strict_inequalities_1a} it follows that, for every $j\in J_1$, the attraction set $A[{\mathbf{x}}^j]$ is $\{{\mathbf{a}}^i\mid i\in I(j)\}$. Since $I(j)\neq\emptyset$ for each $j\in J_1$ and $\mathbf{x}\in F_1(\mathbf{a})$, by Theorem \ref{necessary_optimality_condition} we have \begin{equation}\label{local_solution_components_2}
{\mathbf{x}}^j=\frac{1}{|I(j)|}\displaystyle\sum_{i\in I(j)} {\mathbf{a}}^i.\end{equation} Comparing \eqref{local_solution_components_2} with \eqref{local_solution_components_1a} yields ${\mathbf{x}}^j={\mathbf{x}}^j(\mathbf{a})$ for all $j\in J_1$. By \eqref{strict_inequalities_2a} we see that, for every $j\in J_2$, the attraction set $A[{\mathbf{x}}^j]$ is empty. Moreover, one has 
\begin{equation}\label{abnormal_components_a}
{\mathbf{x}}^j\notin {\mathcal A}[\mathbf{x}]\quad\; (\forall j\in J_2)
\end{equation} where $ {\mathcal A}[\mathbf{x}]$ is the union of the balls $\bar B({\mathbf{a}}^p,\|{\mathbf{a}}^p-{\mathbf{x}}^q\|)$ with $p\in I$, $q\in J$ satisfying $p\in I(q)$.

For each $j\in J_1$, using \eqref{x_j_a} we have
\begin{equation*}\begin{array}{rl}
\|{\mathbf{x}}^j(\widetilde {\mathbf{a}})-\bar {\mathbf{x}}^j\| & \leq \|{\mathbf{x}}^j(\widetilde {\mathbf{a}})-{\mathbf{x}}^j(\mathbf{a})\|+ \|{\mathbf{x}}^j(\mathbf{a})-\bar {\mathbf{x}}^j\|\\
& \leq \frac{1}{k} L_1\|\widetilde {\mathbf{a}}-\mathbf{a}\|+\varepsilon\\
& \leq \frac{1}{k} L_1\left(\|\widetilde {\mathbf{a}}-\bar{\mathbf{a}}\|+\|\bar{\mathbf{a}}-\mathbf{a}\|\right)+\varepsilon\\
& \leq \frac{2}{k} L_1\delta_1+\varepsilon<2\varepsilon.
\end{array}
\end{equation*} Besides, for each  $j\in J_2$, we have $\|{\mathbf{x}}^j(\widetilde {\mathbf{a}})-\bar {\mathbf{x}}^j\| =\|{\mathbf{x}}^j-\bar {\mathbf{x}}^j\|<\varepsilon$.  Therefore,
$$\|\widetilde {\mathbf{x}}-\bar{\mathbf{x}}\|=\sum_{j\in J_1}\|{\mathbf{x}}^j(\widetilde {\mathbf{a}})-\bar {\mathbf{x}}^j\|+\sum_{j\in J_2}\|{\mathbf{x}}^j-\bar {\mathbf{x}}^j\|<2k\varepsilon.$$ In combination with the inequality $\|\widetilde {\mathbf{a}}-\bar{\mathbf{a}} \|<\delta_1$, this assures that the properties \eqref{strict_inequalities_1a} and \eqref{strict_inequalities_2a}, where $\widetilde {\mathbf{a}}$ and  $\mathbf{x}(\widetilde {\mathbf{a}})$ respectively play the roles of $\mathbf{a}$ and $\mathbf{x}$, hold. In other words, one has 
\begin{equation}\label{strict_ineq_tilde_1}
\|\widetilde {\mathbf{a}}^i-\widetilde {\mathbf{x}}^j\|<\|\widetilde {\mathbf{a}}^i-\widetilde {\mathbf{x}}^q\|\quad\; (\forall j\in J_1,\; \forall i\in I(j),\; \forall q\in J\setminus\{j\})\end{equation}
and
\begin{equation}\label{strict_ineq_tilde_2}
\|\widetilde {\mathbf{x}}^q-\widetilde {\mathbf{a}}^p\|<\|\widetilde {\mathbf{x}}^j-\widetilde {\mathbf{a}}^p\|\quad\; (\forall j\in J_2,\; \forall q\in J_1,\; \forall p\in I(q)).\end{equation}
So, similar to the above case of $\mathbf{x}$, for every $j\in J_1$, the attraction set $A[\widetilde {\mathbf{x}}^j]$ is $\{\widetilde {\mathbf{a}}^i\mid i\in I(j)\}$. Recall that $I(j)\neq\emptyset$ for each $j\in J_1$ and $\widetilde {\mathbf{x}}^j$ was given by \begin{equation}\label{local_solution_components_3}
\widetilde {\mathbf{x}}^j={\mathbf{x}}^j(\widetilde {\mathbf{a}})=\frac{1}{|I(j)|}\displaystyle\sum_{i\in I(j)} \widetilde {\mathbf{a}}^i.\end{equation} In addition, for every $j\in J_2$, the attraction set $A[\widetilde {\mathbf{x}}^j]$ is empty and one has 
\begin{equation}\label{abnormal_components_atilde}
\widetilde {\mathbf{x}}^j\notin {\mathcal A}[\widetilde{\mathbf{x}}]\quad\; (\forall j\in J_2),
\end{equation} where $ {\mathcal A}[\widetilde{\mathbf{x}}]$ is the union of the balls $\bar B(\widetilde {\mathbf{a}}^p,\|\widetilde {\mathbf{a}}^p-\widetilde {\mathbf{x}}^q\|)$ with $p\in I$, $q\in J$ satisfying $p\in I(q)$. Besides, from \eqref{strict_ineq_tilde_1} and \eqref{strict_ineq_tilde_2} it follows that  $|J_i(\widetilde {\mathbf{x}})|=1$ for every $i\in I$. Since $\|\widetilde {\mathbf{x}}-\bar{\mathbf{x}}\|<2k\varepsilon$, we have  $\widetilde {\mathbf{x}}^{j_1}\neq\widetilde {\mathbf{x}}^{j_2}$ for all $j_1,j_2\in J$ with $j_2\neq j_1$. Due to the last two properties and \eqref{local_solution_components_3}, \eqref{abnormal_components_atilde}, by Theorem \ref{sufficient_optimality_condition} we conclude that $\widetilde {\mathbf{x}}\in F_1(\widetilde {\mathbf{a}})$.

{\sc Claim 2.} \textit{One has $\mathbf{x}\in \widetilde {\mathbf{x}}+ L_1 \|\mathbf{a}-\widetilde{\mathbf{a}}\|\bar B_{\mathbb R^{n\times k}}$.}

Indeed, since ${\mathbf{x}}^j={\mathbf{x}}^j(\mathbf{a})$ for all $j\in J_1$, $\widetilde {\mathbf{x}}^j={\mathbf{x}}^j$ for any $j\in J_2$, by \eqref{local_solution_components_3} and \eqref{x_j_a} we have  
\begin{align*}
\label{x_x_widetilde}
\|\mathbf{x}-\widetilde {\mathbf{x}}\|&= \sum_{j\in J_1} \|{\mathbf{x}}^j-\widetilde {\mathbf{x}}^j\|+ \sum_{j\in J_2} \|{\mathbf{x}}^j-\widetilde {\mathbf{x}}^j\|\\
& = \sum_{j\in J_1}  \|{\mathbf{x}}^j(\mathbf{a})- {\mathbf{x}}^j(\widetilde {\mathbf{a}})\|\\
&\leq k.\frac{1}{k} L_1\|\mathbf{a}-\widetilde {\mathbf{a}}\|=L_1\|\mathbf{a}-\widetilde {\mathbf{a}}\|.
\end{align*}
It follows that ${\mathbf{x}}\in \widetilde{\mathbf{x}}+ L_1 \|\mathbf{a}-\widetilde{\mathbf{a}}\|\bar B_{\mathbb R^{n\times k}}$. 

Combining Claim 2 with Claim 1, we have ${\mathbf{x}}\in F_1(\widetilde {\mathbf{a}})+L_1 \|\mathbf{a}-\widetilde{\mathbf{a}}\|\bar B_{\mathbb R^{n\times k}}$. Thus, property \eqref{Aubin property} is valid for all $\mathbf{a}$ and $\widetilde  {\mathbf{a}}$ satisfying $\|\mathbf{a}-\bar{\mathbf{a}}\|<\delta_1$ and $\|\widetilde {\mathbf{a}}-\bar{\mathbf{a}}\|<\delta_1$. $\hfill\Box$
\end{proof}

\end{document}